\documentclass{amsart}
\usepackage{amssymb}
\usepackage{amsmath}
\usepackage{amsfonts}

\newtheorem{theorem}{Theorem}
\theoremstyle{plain}

\newtheorem{corollary}{Corollary}

\newtheorem{lemma}{Lemma}

\newtheorem{proposition}{Proposition}
\newtheorem{remark}{Remark}

\numberwithin{equation}{section}
\input{tcilatex}

\begin{document}
\title[On the geometry of the rescaled Riemannian metric]{On the geometry of
the rescaled Riemannian metric on tensor bundles of arbitrary type}
\author{Aydin GEZER}
\address{Ataturk University, Faculty of Science, Department of Mathematics,
25240, Erzurum-Turkey.}
\email{agezer@atauni.edu.tr}
\author{Murat ALTUNBAS}
\address{Erzincan University, Faculty of Science and Art, Department of
Mathematics, 24030, Erzincan-Turkey.}
\email{maltunbas@erzincan.edu.tr}

\begin{abstract}
Let $(M,g)$ be an $n-$dimensional Riemannian manifold and $T_{1}^{1}(M)$ be
its $(1,1)-$tensor bundle equipped with the rescaled Sasaki type metric $%
^{S}g_{f}$ which rescale the horizontal part by a nonzero differentiable
function $f$. In the present paper, we discuss curvature properties of the
Levi-Civita connection and another metric connection of $T_{1}^{1}(M)$. We
construct almost paracomplex Norden structures on $T_{1}^{1}(M)$ and
investigate conditions for these structures to be para-K\"{a}hler
(paraholomorphic) and quasi-K\"{a}hler. Also, some properties of almost
paracomplex Norden structures in context of almost product Riemannian
manifolds are presented. Finally we introduce the rescaled Sasaki type
metric $^{S}g_{f}$ on the $(p,q)-$\ tensor bundle and characterize the
geodesics on the $(p,q)$-tensor bundle with respect to the Levi-Civita
connection of \textit{${}$}$^{S}g_{f}$ and another metric connection of 
\textit{${}$}$^{S}g_{f}.$
\end{abstract}

\maketitle

\section{\protect\bigskip \textbf{Introduction}}

\noindent Geometric structures on bundles have been object of much study
since the middle of the last century. The natural lifts of the metric $g$,
from a Riemannian manifold $(M,g)$ to its tangent or cotangent bundles,
induce new (pseudo) Riemannian structures, with interesting geometric
properties. Maybe the best known Riemannian metric $^{S}g$ on the tangent
bundle over Riemannian manifold $(M,g)$ is that introduced by S. Sasaki in
1958 (see \cite{Sasaki}), but in most cases the study of some geometric
properties of the tangent bundle endowed with this metric led to the
flatness of the base manifold. This metric $^{S}g$ is a standard notion in
differential geometry called the Sasaki metric (for the recent surveys on
the Sasaki metric, see \cite{GudmundssonKappos:2002}). The Sasaki metric $%
^{S}g$ has been extensively studied by several authors and in many different
contexts. In \cite{Zayatuev1}(see also \cite{Zayatuev2,Zayatuev3}, B. V.
Zayatuev introduced a Riemannian metric $^{S}g_{f}$ on the tangent bundle $%
TM $ given by 
\begin{eqnarray*}
^{S}g_{f}\left( ^{H}X,^{H}Y\right) &=&fg\left( X,Y\right) , \\
^{S}g_{f}\left( ^{H}X,^{V}Y\right) &=&^{S}g_{f}\left( ^{V}X,^{H}Y\right) =0,
\\
^{S}g_{f}\left( ^{V}X,^{V}Y\right) &=&g\left( X,Y\right) ,
\end{eqnarray*}%
for all vector fields $X$ and $Y$ on $M$, where $f>0$, $f\in C^{\infty }(M).$
For $f=1$, it follows that $^{S}g_{f}=^{S}g,$ i.e. the metric $^{S}g_{f}$ is
a generalization of the Sasaki metric $^{S}g$. In \cite{Wang}, J. Wang and
Y. Wang called this metric the rescaled Sasaki metric and studied geodesics
and some curvature properties for the rescaled Sasaki metric. Also, the
authors studied the rescaled Sasaki type metric on the cotangent bundle $%
T^{\ast }M$ over Riemannian manifold $(M,g)$ (see \cite{Gezer2}).

Almost complex Norden and almost paracomplex Norden structures are among the
most important geometrical structures which can be considered on a manifold.
Let $M_{2k}$ be a $2k$-dimensional differentiable manifold endowed with an
almost (para) complex structure $\varphi $ and a pseudo-Riemannian metric $g$
of signature $(k,k)$ such that $g(\varphi X,Y)=g(X,\varphi Y)$, i.e. $g$ is
pure with respect to $\varphi $ for arbitrary vector fields $X$ and $Y$ on $%
M_{2k}$. Then the metric $g$ is called Norden metric. Norden metrics are
referred to as anti-Hermitian metrics or $B$-metrics. They find widespread
application in mathematics as well as in theoretical physics. Many authors
considered almost (para) complex Norden structures on the tangent, cotangent
and tensor bundles \cite%
{Druta1,Gezer1,Olszak,Oproiu1,Oproiu2,Oproiu3,Oproiu4,Papag1,Papag2,Salimov2,Salimov3}%
.

Fibre bundles play an important role in just about every aspect of modern
geometry and topology. Prime examples of fiber bundles are tensor bundles of
arbitrary type over differentiable manifolds. The tangent bundle $TM$ and
cotangent bundle $T^{\ast }M$ are the special cases of a more general tensor
bundle. The Sasaki type metric is defined on $(p,q)-$tensor bundles over
Riemannian manifolds (see, \cite{Salimov5}). In \cite{Salimov3}, the
Levi-Civita connection of the Sasaki type metric on the $(1,1)-$tensor
bundle and all types of its curvature tensors are calculated and also
investigated interesting relations between the geometric properties of the
base manifold and its $(1,1)-$tensor bundle with the Sasaki type metric. In
addition, it is presented examples of almost para-Norden and para-K\"{a}%
hler-Norden $B$ -metrics on the $(1,1)-$tensor bundle with the Sasaki type
metric.

Motivated by the above studies, our aim is to define the rescaled Sasaki
type metric on tensor bundles of arbitrary type and study its some
properties. The paper is organized as follows. In section 2, we review some
introductory materials concerning with the tensor bundle $T_{1}^{1}(M)$ over
an $n$-dimensional Riemannian manifold $M.$ In section 3, we get the
conditions under which the tensor bundle $T_{1}^{1}(M)$ endowed with some
paracomplex structures and the rescaled Sasaki type $^{S}g_{f}$ is a
paraholomorphic Norden manifold. In addition, for an almost paracomplex
manifold to be an specialized almost product manifold, we give some results
relation to Riemannian almost product structure on the tensor bundle $%
T_{1}^{1}(M)$. Section 4 and section 5 discuss curvature properties of the
Levi-Civita connection and another metric connection of $^{S}g_{f}.$ Section
6 deals with detailed descriptions of geodesics on $(p,q)$-tensor bundles
with respect to the Levi-Civita connection of \textit{${}$}$^{S}g_{f}$ and
another metric connection of \textit{${}$}$^{S}g_{f}.$

Throughout this paper, all manifolds, tensor fields and connections are
always assumed to be differentiable of class $C^{\infty }$. Also, we denote
by $\Im _{q}^{p}(M)$ the set of all tensor fields of type $(p,q)$ on $M$,
and by $\Im _{q}^{p}(T_{q}^{p}(M))$ the corresponding set on the $(p,q)$%
-tensor bundle $T_{q}^{p}(M)$. The Einstein summation convention is used,
the range of the indices $i,j,s$ being always $\{1,2,...,n\}.$

\section{Preliminaries}

\subsection{The $(1,1)-$tensor bundle}

Let $M$ be a differentiable manifold of class $C^{\infty }$ and finite
dimension $n$. Then the set $T_{1}^{1}(M)=\cup _{P\in M}T_{1}^{1}(P)$ is, by
definition, the tensor bundle of type $(1,1)$ over $M$, where $\bigcup $
denotes the disjoint union of the tensor spaces $T_{1}^{1}(P)$ for all $P\in
M$. For any point $\tilde{P}$ of $T_{1}^{1}(M)$ such that $\tilde{P}\in
T_{1}^{1}(M)$, the surjective correspondence $\tilde{P}\rightarrow P$
determines the natural projection $\pi :T_{1}^{1}(M)\rightarrow M$. The
projection $\pi $ defines the natural differentiable manifold structure of $%
T_{1}^{1}(M)$, that is, $T_{1}^{1}(M)$ is a $C^{\infty }$-manifold of
dimension $n+n^{2}$. If $x^{j}$ are local coordinates in a neighborhood $U$
of $P\in M$, then a tensor $t$ at $P$ which is an element of $T_{1}^{1}(M)$
is expressible in the form $(x^{j},t_{j}^{i})$, where $t_{j}^{i}$ are
components of $t$ with respect to the natural base. We may consider $%
(x^{j},t_{j}^{i})=(x^{j},x^{\bar{j}})=(x^{J})$, $j=1,...,n$, $\bar{j}%
=n+1,...,n+n^{2}$, $J=1,...,n+n^{2}$ as local coordinates in a neighborhood $%
\pi ^{-1}(U)$.

Let $X=X^{i}\frac{\partial }{\partial x^{i}}$ and $A=A_{j}^{i}\frac{\partial 
}{\partial x^{i}}\otimes dx^{i}$ be the local expressions in $U$ of a vector
field $X$ \ and a $(1,1)$ tensor field $A$ on $M$, respectively. Then the
vertical lift $^{V}A$ of $A$ and the horizontal lift $^{H}X$ of $X$ are
given, with respect to the induced coordinates, by\noindent 
\begin{equation}
{}^{V}A=\left( 
\begin{array}{l}
{{}^{V}A^{j}} \\ 
{{}^{V}A^{\overline{j}}}%
\end{array}%
\right) =\left( 
\begin{array}{c}
{0} \\ 
{A_{j}^{i}}%
\end{array}%
\right) ,  \label{A2.1}
\end{equation}%
and

\begin{equation}
^{H}X=\left( 
\begin{array}{c}
{{}^{H}X^{j}} \\ 
{{}^{H}X^{\overline{j}}}%
\end{array}%
\right) =\left( 
\begin{array}{c}
{V^{j}} \\ 
X{^{s}(\Gamma _{sj}^{m}t_{m}^{i}-\Gamma _{sm}^{i}t_{j}^{m})}%
\end{array}%
\right) ,  \label{A2.2}
\end{equation}%
where $\Gamma _{ij}^{h}$ are the coefficients of the Levi-Civita connection $%
\nabla $ of $g$.

Let $\varphi \in \Im _{1}^{1}(M)$, which are locally represented by $\varphi
=\varphi _{j}^{i}\frac{\partial }{\partial x^{i}}\otimes dx^{j}$. The vector
fields $\gamma \varphi $ and $\tilde{\gamma}\varphi \in \Im
_{0}^{1}(T_{1}^{1}(M))$ are respectively defined by

\begin{equation*}
\left\{ 
\begin{array}{l}
{\gamma \varphi =\left( 
\begin{array}{c}
{0} \\ 
{t_{j}^{m}\varphi _{m}^{i}}%
\end{array}%
\right) ,} \\ 
{\tilde{\gamma}\varphi =\left( 
\begin{array}{c}
{0} \\ 
{(t_{m}^{i}\varphi _{j}^{m})}%
\end{array}%
\right) ,}%
\end{array}%
\right.
\end{equation*}%
with respect to the coordinates $(x^{j},x^{\bar{j}})$ in $T_{1}^{1}(M)$.
From (\ref{A2.1}) we easily see that the vector fields $\gamma \varphi $ and 
$\tilde{\gamma}\varphi $ determine respectively global vector fields on $%
T_{1}^{1}(M)$.

The Lie bracket operation of vertical and horizontal vector fields on $%
T_{1}^{1}(M)$ is given by the formulas%
\begin{equation}
\left\{ 
\begin{array}{l}
{\left[ {}^{H}X,{}^{H}Y\right] ={}^{H}\left[ X,Y\right] +(\tilde{\gamma}%
-\gamma )R(X},{Y}), \\ 
{\left[ {}^{H}X,{}^{V}A\right] ={}^{V}(\nabla _{X}A),} \\ 
{\left[ {}^{V}A,{}^{V}B\right] =0}%
\end{array}%
\right.  \label{A2.3}
\end{equation}%
for any $X,$ $Y$ $\in \Im _{0}^{1}(M)$ and $A$, $B\in \Im _{1}^{1}(M)$,
where $R$ is the Riemannian curvature of $g$ defined by $R\left( X,Y\right) =%
\left[ \nabla _{X},\nabla _{Y}\right] -\nabla _{\left[ X,Y\right] }$ \ and ${%
(\tilde{\gamma}-\gamma )R(X},{Y})=\left( 
\begin{array}{c}
0 \\ 
{t_{m}^{i}R}_{klj}^{\text{ \ \ \ }m}X^{k}Y^{l}{-t_{j}^{m}R}_{klm}^{\text{ \
\ \ }i}X^{k}Y^{l}%
\end{array}%
\right) $ (for details, see \cite{Cengiz,Salimov3,Salimov5})$.$

\subsection{Expressions in the adapted frame}

We insert the adapted frame which allows the tensor calculus to be
efficiently done in $T_{1}^{1}(M).$ With the connection $\nabla $ of $g$ on $%
M$, we can introduce adapted frames on each induced coordinate neighborhood $%
\pi ^{-1}(U)$ of $T_{1}^{1}(M)$. In each local chart $U\subset M$, we write $%
X_{(j)}=\partial _{h}=\delta _{j}^{h}\partial _{h}\in \Im _{0}^{1}(M),A^{(%
\overline{j})}=\partial _{i}\otimes dx^{j}=\delta _{i}^{k}\delta
_{h}^{j}\partial _{k}\otimes dx^{h}\in \Im _{1}^{1}(M),$ $j=1,...,n,$ $%
\overline{j}=n+1,...,n+n^{2}.$ Then from (\ref{A2.1}) and (\ref{A2.2}), we
see that these vector fields have respectively local expressions 
\begin{equation*}
^{H}X_{(j)}=\delta _{j}^{h}\partial _{h}+(-t_{h}^{s}\Gamma
_{js}^{k}+t_{s}^{k}\Gamma _{jh}^{s})\partial _{\overline{h}}
\end{equation*}%
\begin{equation*}
^{V}A^{(\overline{j})}=\delta _{i}^{k}\delta _{h}^{j}\partial _{\overline{h}}
\end{equation*}%
with respect to the natural frame $\left\{ \partial _{h},\partial _{%
\overline{h}}\right\} $ in $T_{1}^{1}(M),$ where $\partial _{h}=\frac{%
\partial }{\partial x^{h}}$, $\partial _{\overline{h}}=\frac{\partial }{%
\partial x^{\overline{h}}}$, $x^{\overline{h}}=t_{h}^{k}$ and $\delta
_{i}^{j}$ is the Kronecker's. These $n+n^{2}$ vector fields are linearly
independent and they generate the horizontal distribution of $\nabla _{g}$
and the vertical distribution of $T_{1}^{1}(M)$, respectively. The set $%
\left\{ ^{H}X_{(j)},^{V}A^{(\overline{j})}\right\} $ is called the frame
adapted to the connection $\nabla _{g}$ of $g$ in $\pi ^{-1}(U)\subset
T_{1}^{1}(M)$. By denoting%
\begin{eqnarray}
E_{j} &=&^{H}X_{(j)},  \label{A2.4} \\
E_{\overline{j}} &=&^{V}A^{(\overline{j})},  \notag
\end{eqnarray}%
we can write the adapted frame as $\left\{ E_{\alpha }\right\} =\left\{
E_{j},E_{\overline{j}}\right\} $. The indices $\alpha ,\beta ,\gamma
,...=1,...,n+n^{2}$ indicate the indices with respect to the adapted frame.

Using (\ref{A2.1}), (\ref{A2.2}) and (\ref{A2.4}), we have%
\begin{equation}
^{V}A=\left( 
\begin{array}{l}
0 \\ 
A_{j}^{i}%
\end{array}%
\right) ,  \label{A2.5}
\end{equation}%
and

\begin{equation}
^{H}X=\left( 
\begin{array}{l}
X^{j} \\ 
0%
\end{array}%
\right)  \label{A2.6}
\end{equation}%
with respect to the adapted frame $\left\{ E_{\alpha }\right\} $ (for
details, see \cite{Salimov3}). By the straightforward calculations, we have
the lemma below.

\begin{lemma}
\label{Lemma1}The Lie brackets of the adapted frame of $T_{1}^{1}(M)$
satisfy the following identities:%
\begin{eqnarray*}
\left[ E_{l},E_{j}\right] &=&(t_{s}^{v}R_{ijr}^{\text{ \ \ }%
s}-t_{r}^{s}R_{ljs}^{\text{ \ \ }v})E_{\overline{r}}, \\
\left[ E_{l},E_{\overline{j}}\right] &=&(\delta _{r}^{j}\Gamma
_{li}^{v}-\delta _{i}^{v}\Gamma _{lr}^{j})E_{\overline{r}}, \\
\left[ E_{\overline{i}},E_{\overline{j}}\right] &=&0
\end{eqnarray*}

where $R_{ijl}^{\text{ \ \ }s}$ denote the components of the curvature
tensor of $(M,g)$.\bigskip
\end{lemma}

\section{Almost paracomplex structures with Norden metrics}

\bigskip Let $T_{1}^{1}(M)$ be the $(1,1)-$tensor bundle over a Riemannian
manifold $(M,g)$. For each $P\in M,$ the extension of scalar product $g$
(marked by $G$) is defined on the tensor space $\pi ^{-1}(P)=T_{1}^{1}(P)$
by $G(A,B)=g_{it}g^{jl}A_{j}^{i}B_{l}^{t}$ for all $A,$ $B\in \Im
_{1}^{1}\left( P\right) $. The rescaled Sasaki type metric $^{S}g_{f}$ is
defined on $T_{1}^{1}(M)$ by the following three equations%
\begin{equation}
^{S}g_{f}\left( {}^{V}A,{}^{V}B\right) ={}^{V}(G(A,B)),  \label{B3.1}
\end{equation}%
\begin{equation}
{}^{S}g_{f}\left( {}^{V}A,{}^{H}Y\right) =0,  \label{B3.2}
\end{equation}%
\begin{equation}
{}^{S}g_{f}\left( {}^{H}X,{}^{H}Y\right) =\text{ }^{V}({}fg\left( X,Y\right)
)  \label{B3.3}
\end{equation}%
for any $X,Y\in \Im _{0}^{1}\left( M\right) $ and $A,B\in \Im _{1}^{1}\left(
M\right) $, where $f>0$, $f\in C^{\infty }(M)$ (for $f=1,$ see \cite%
{Salimov3}). From the equations (\ref{B3.1})-(\ref{B3.3}), by virtue of (\ref%
{A2.5}) and (\ref{A2.6}), the rescaled Sasaki type metric $^{S}g_{f}$ and
its inverse have components with respect to the adapted frame $\left\{
E_{\alpha }\right\} $:%
\begin{equation}
(^{S}g_{f})_{\beta \gamma }=\left( 
\begin{array}{cc}
(^{S}g_{f})_{jl} & (^{S}g_{f})_{j\overline{l}} \\ 
(^{S}g_{f})_{\overline{j}l} & (^{S}g_{f})_{\overline{j}\overline{l}}%
\end{array}%
\right) =\left( 
\begin{array}{cc}
{fg_{jl}} & {0} \\ 
{0} & g_{it}g^{jl}%
\end{array}%
\right) \text{ , }x^{\overline{l}}=t_{l}^{t}  \label{B3.4}
\end{equation}%
and%
\begin{equation}
(^{S}g_{f})^{\beta \gamma }=\left( 
\begin{array}{cc}
(^{S}g_{f})^{jl} & (^{S}g_{f})^{j\overline{l}} \\ 
(^{S}g_{f})^{\overline{j}l} & (^{S}g_{f})^{\overline{j}\overline{l}}%
\end{array}%
\right) =\left( 
\begin{array}{cc}
\frac{1}{f}{g}^{jl} & {0} \\ 
{0} & g^{it}g_{jl}%
\end{array}%
\right) \text{ , }x^{\overline{j}}=t_{j}^{i}  \label{B3.5}
\end{equation}

For the Levi-Civita connection of the rescaled Sasaki type metric $^{S}g_{f}$
we give the next theorem.

\begin{theorem}
\label{Theorem1} Let $(M,g)$ be a Riemannian manifold and equip its tensor
bundle $T_{1}^{1}(M)$ with the rescaled Sasaki type metric $^{S}g_{f}$. Then
the corresponding Levi-Civita connection $\widetilde{\nabla }$ satisfies the
followings:%
\begin{equation}
\left\{ 
\begin{array}{l}
i)\text{ }{\widetilde{\nabla }_{E_{l}}E_{j}=\{\Gamma _{lj}^{r}+\frac{1}{2f}{}%
^{f}A_{lj}^{r}\}E_{r}+\{\frac{1}{2}R_{ljr}^{\text{ \ \ }s}t_{s}^{v}-\frac{1}{%
2}R_{ljs}^{\text{ \ \ }v}t_{r}^{s}\}E_{\overline{r}},} \\ 
ii)\text{ }{\widetilde{\nabla }_{E_{l}}E_{\overline{j}}=\{\frac{1}{2f}%
g_{ia}R_{.}^{s}{}_{.}^{j}{}_{l}^{\text{ }r}t_{s}^{a}-\frac{1}{2f}%
g^{jb}R_{isl}^{\text{ \ \ }r}t_{b}^{s}\}E_{r}+\{\Gamma _{li}^{v}\delta
_{r}^{j}-\Gamma _{lr}^{j}\delta _{i}^{v}\}E_{\overline{r}},} \\ 
iii)\text{ }{\widetilde{\nabla }_{E_{\overline{l}}}E_{j}=\{\frac{1}{2f}%
g_{ta}R_{.}^{s}{}_{.}^{l}{}_{j}^{\text{ }r}t_{s}^{a}-\frac{1}{2f}%
g^{lb}R_{tsj}^{\text{ \ \ }r}t_{b}^{s}\}E_{r},} \\ 
iv)\text{ }{\widetilde{\nabla }_{E_{\overline{l}}}E_{\overline{j}}=0}%
\end{array}%
\right.  \label{B3.6}
\end{equation}%
with respect to the adapted frame, where ${}^{f}A_{ji}^{h}$ is a symmetric
tensor field of type $(1,2)$ defined by ${}^{f}A_{ji}^{h}=(f_{j}\delta
_{i}^{h}+f_{i}\delta _{j}^{h}-f_{.}^{h}g_{ji})$ and $f_{i}=\partial _{i}f$, $%
{R_{.}^{s}{}_{.}^{j}{}_{l}^{\text{ }r}}=g^{as}g^{bj}R_{abl}^{\text{ \ \ \ \ }%
r}.$
\end{theorem}

\begin{proof}
The connection \bigskip ${}\widetilde{{\nabla }}$ is characterized by the
Koszul formula:%
\begin{eqnarray*}
2^{S}g_{f}({}\widetilde{{\nabla }}_{\widetilde{X}}\widetilde{Y},\widetilde{Z}%
) &=&\widetilde{X}(^{S}g_{f}(\widetilde{Y},\widetilde{Z}))+\widetilde{Y}%
(^{S}g_{f}(\widetilde{Z},\widetilde{X}))-\widetilde{Z}(^{S}g_{f}(\widetilde{X%
},\widetilde{Y})) \\
-^{S}g_{f}(\widetilde{X},[\widetilde{Y},\widetilde{Z}]) &+&^{S}g_{f}(%
\widetilde{Y},[\widetilde{Z},\widetilde{X}])+\text{ }^{S}g_{f}(\widetilde{Z}%
,[\widetilde{X},\widetilde{Y}])
\end{eqnarray*}%
for all vector fields $\widetilde{X},\widetilde{Y}$ and $\widetilde{Z}$ on $%
T_{1}^{1}(M)$. One can verify the Koszul formula for pairs $\widetilde{X}=$ $%
E_{l},E_{\overline{l}}$ and $\widetilde{Y}=$ $E_{j},E_{\overline{j}}$ and $%
\widetilde{Z}=$ $E_{k},E_{\overline{k}}$. In calculations, the formulas (\ref%
{A2.4}), Lemma \ref{Lemma1} and the first Bianchi identity for $R$ should be
applied. We omit standart calculations.
\end{proof}

Let $\widetilde{X},\;\widetilde{Y}\in \Im _{0}^{1}(T_{1}^{1}(M))$. Then the
covariant derivative $\widetilde{{\nabla }}_{\widetilde{X}}\widetilde{Y}$
has components

\begin{equation*}
\widetilde{{\nabla }}_{\widetilde{X}}\widetilde{Y}^{\alpha }=\widetilde{X}%
^{\gamma }E_{\gamma }\widetilde{Y}^{\alpha }+\widetilde{\Gamma }_{\gamma
\beta }^{\alpha }\widetilde{X}^{\beta }\widetilde{Y}^{\gamma }
\end{equation*}%
with respect to the adapted frame $\left\{ E_{\alpha }\right\} $. \noindent
Using (\ref{A2.4}), (\ref{A2.5}), (\ref{A2.6}) and (\ref{B3.6}), we have the
following proposition.

\begin{proposition}
\label{Proposition1}Let $(M,g)$ be a Riemannian manifold\textit{\ and }${}%
\widetilde{\nabla }$\textit{\ be the Levi-Civita connection of the tensor
bundle }$T_{1}^{1}(M)$\textit{\ equipped with the rescaled Sasaki type
metric ${}$}$^{S}g_{f}$\textit{. }Then the corresponding Levi-Civita
connection satisfies the following relations: 
\begin{equation*}
\begin{array}{l}
{i)\mathrm{\;\;}{}}\widetilde{{\nabla }}{_{{}^{H}X}{}^{H}Y={}^{H}\left(
\nabla _{X}Y+\frac{1}{2f}^{f}A(X,Y)\right) +{\frac{1}{2}}(\tilde{\gamma}%
-\gamma )R(X,Y),} \\ 
{ii)\mathrm{\;\;}{}}\widetilde{{\nabla }}{_{{}^{H}X}{}^{V}B={}{\frac{1}{2f}}%
{}^{H}\left( g^{bj}\,R(t_{b},B_{j})X+g_{ai}\,(t^{a}(g^{-1}\circ R(\quad ,X)%
\tilde{B}^{\,i}\right) +{}^{V}\left( \nabla _{X}B\right) ,} \\ 
{iii)\mathrm{\;\;}}\widetilde{{\nabla }}{_{{}^{V}C}{}^{H}Y={\frac{1}{2f}}%
{}^{H}\left( g^{bl}\,R(t_{b},C_{l})Y+g_{at}(t^{a}\,(g^{-1}\circ R(\quad ,Y)%
\widetilde{C}^{\,t})\right) ,} \\ 
{iv)\mathrm{\;\;\;}{}}\widetilde{{\nabla }}{_{{}^{V}C}{}^{V}B=0}%
\end{array}%
\end{equation*}%
\noindent for all $X,\;Y\in \Im _{0}^{1}(M)$ and $B,\;C\in \Im _{1}^{1}(M)$,
where $C_{l}=(C_{l}^{\,\,i})$, $\widetilde{C}^{t}=(g^{bl}C_{l}^{\quad
t})=(C_{\,.}^{b\,t})$, $t_{l}=(t_{l}^{\;a})$, $t^{a}=(t_{b}^{\;a})$, $%
R(\;\;,X)Y\in \Im _{1}^{1}(M),$ $g^{-1}\circ R(\;\;,X)Y\in \Im _{0}^{1}(M)$
and $^{f}A(X,Y)=X(f)Y+Y(f)X-g(X,Y)\circ (df)^{\ast }$\textit{\ (for }$f=1,$
see \cite{Salimov3}).\bigskip
\end{proposition}

An almost paracomplex manifold is an almost product manifold $%
(M_{2k},\varphi )$, $\varphi ^{2}=id$, $\varphi \neq \pm id$ such that the
two eigenbundles $T^{+}M_{2k}$ and $T^{-}M_{2k}$ associated to the two
eigenvalues +1 and -1 of $\varphi $, respectively, have the same rank. Note
that the dimension of an almost paracomplex manifold is necessarily even.
This structure is said to be integrable if the matrix $\varphi =(\varphi
_{j}^{i})$ is reduced to the constant form in a certain holonomic natural
frame in a neighborhood $U_{x}$ of every point $x\in M_{2k}$. On the other
hand, an almost paracomplex structure is integrable if and only if one can
introduce a torsion-free linear connection such that $\nabla \varphi =0$.
Also it can be give another-equivalent-definition of paracomplex manifold in
terms of local homeomorphisms in the space $R^{k}(j)=\left\{
(X^{1},...,X^{k})/X^{i}\in R(j),\mathrm{\;}i=1,...,k\right\} $ and
paraholomorphic changes of charts in a way similar to \cite{CruForGad:1995}
(see also \cite{Vish:DiffGeo}), i.e. a manifold $M_{2k}$ with an integrable
paracomplex structure $\varphi $ is a realization of the paraholomorphic
manifold $M_{k}(R(j))$ over the algebra $R(j)$.

A $(0,q)$ tensor field $\omega $ is called a pure tensor field with respect
to $\varphi $ if 
\begin{equation*}
\omega (\varphi X_{1},X_{2},...,X_{q})=\omega (X_{1},\varphi
X_{2},...,X_{q})=...=\omega (X_{1},X_{2},...,\varphi X_{q})
\end{equation*}%
for any $X_{1},...,X_{q}\in \Im _{0}^{1}(M_{2k}).$ The real model of a
paracomplex tensor field $\mathop{\omega }\limits^{\ast }$ on $M_{k}(R(j))$
is a $(0,q)$ tensor field on $M_{2k}$ which being pure with respect to $%
\varphi $. Consider the operator $\phi _{\varphi }:\Im
_{q}^{0}(M_{2k})\rightarrow \Im _{q+1}^{0}(M_{2k})$ applied to the pure
tensor field $\omega $ by (see \cite{YanoAko:1968})%
\begin{equation*}
(\phi _{\varphi }\omega )(X,Y_{1},Y_{2},...,Y_{q})=(\varphi X)(\omega
(Y_{1},Y_{2},...,Y_{q}))-X(\omega (\varphi Y_{1},Y_{2},...,Y_{q}))
\end{equation*}%
\begin{equation*}
+\omega ((L_{Y_{1}}\varphi )X,Y_{2},...,Y_{q})+...+\omega
(Y_{1},Y_{2},...,(L_{Y_{q}}\varphi )X),
\end{equation*}%
where $L_{Y}$ denotes the Lie differentiation with respect to $Y$. \noindent
Let $\varphi $ be a (an almost) paracomplex structure on $M_{2k}$ and $\phi
_{\varphi }\omega =0$, the (almost) paracomplex tensor field $%
\mathop{\omega
}\limits^{\ast }$ on $M_{k}(R(j))$ is said to be (almost) paraholomorphic
(see \cite{Kruchkovich:1972}, \cite{Tachibana}, \cite{YanoAko:1968}). Thus a
(an almost) paraholomorphic tensor field $\mathop{\omega }\limits^{\ast }$
on $M_{k}(R(j))$ is realized on $M_{2k}$ in the form of a pure tensor field $%
\omega $, such that%
\begin{equation*}
(\phi _{\varphi }\omega )(X,Y_{1},Y_{2},...,Y_{q})=0
\end{equation*}%
for any $X,Y_{1},...,Y_{q}\in \Im _{0}^{1}(M_{2k})$. Therefore, the tensor
field $\omega $ on $M_{2k}$ is also called a (an almost) paraholomorphic
tensor field.

An almost paracomplex Norden manifold $(M_{2k},\varphi ,g)$ is defined to be
a real differentiable manifold $M_{2k}$ endowed with an almost paracomplex
structure $\varphi $ and a Riemannian metric $g$ satisfying Nordenian
property (or purity condition)%
\begin{equation*}
g(\varphi X,Y)=g(X,\varphi Y)
\end{equation*}%
for any $X,Y\in \Im _{0}^{1}(M_{2k})$. Almost paracomplex manifolds with
Norden metrics are referred to as anti-Hermitian and $B$-manifolds. If $%
\varphi $ is integrable, we say that $(M_{2k},\varphi ,g)$ is a paracomplex
Norden manifold.

In a paracomplex Norden manifold, a paracomplex Norden metric $g$ is called
paraholomorphic\textit{\ }if%
\begin{equation}
(\phi _{\varphi }g)(X,Y,Z)=0  \label{B3.7}
\end{equation}%
for any $X,Y,Z\in \Im _{0}^{1}(M_{2k})$. \noindent The paracomplex Norden
manifold $(M_{2k},\varphi ,g)$ with a paraholomorphic Norden metric is
called a paraholomorphic Norden manifold.

In \cite{SalimovIscanEtayo:2007}, A. A. Salimov and his collaborators have
been proven that for an almost paracomplex manifold with Norden metric $g$,
the condition $\phi _{\varphi }g=0$ is equivalent to $\nabla \varphi =0$,
where $\nabla $ is the Levi-Civita connection of $g$. By virtue of this
point of view, paraholomorphic Norden manifolds are similar to K\"{a}hler
manifolds. A paraholomorphic Norden manifold is also called a para-K\"{a}%
hler-Norden manifold.

Let us define an almost paracomplex structure on $T_{1}^{1}(M)$ as follows:%
\begin{eqnarray}
J(^{H}X) &=&-^{H}X  \label{B3.8} \\
J({}^{V}A) &=&^{V}A  \notag
\end{eqnarray}%
for any $X\in \Im _{0}^{1}\left( M\right) $ and $A\in \Im _{1}^{1}\left(
M\right) $. One can easily check that the rescaled Sasaki type metric $%
^{S}g_{f}$ is pure with respect to the almost paracomplex structure $J$.
Hence we state the following theorem.

\begin{theorem}
\label{Theorem2}Let $(M,g)$ be a Riemannian manifold and $T_{1}^{1}(M)$ be
its tensor bundle equipped with the rescaled Sasaki type metric $^{S}g_{f}$
and the paracomplex structure $J$. The triple $(T_{1}^{1}(M),J,^{S}g_{f})$
is an almost paracomplex Norden manifold.
\end{theorem}

We now give conditions for the rescaled Sasaki type metric $^{S}g_{f}$ to be
paraholomorphic with respect to the almost paracomplex structure $J$. Using
the definition of the rescaled Sasaki type metric $^{S}g_{f}$ and the almost
paracomplex structure $J$ and by using the fact that $%
{}^{V}A{}^{V}(G(B,C))=0,$ $^{V}A^{V}(fg(Y,Z))=0$ and $%
{}^{H}X{}^{V}(fg(Y,Z))={}^{V}(X(fg(Y,Z)))$ we calculate 
\begin{eqnarray*}
(\phi _{J}{}^{S}g_{f})(\tilde{X},\tilde{Y},\tilde{Z}) &=&(J\tilde{X}%
)({}^{S}g_{f}(\tilde{Y},\tilde{Z}))-\tilde{X}(^{S}g_{f}(J\tilde{Y},\tilde{Z}%
)) \\
&+&{}^{S}g_{f}((L_{\tilde{Y}}J)\tilde{X},\tilde{Z})+{}^{S}g_{f}(\tilde{Y}%
,(L_{\tilde{Z}}J)\tilde{X})
\end{eqnarray*}%
for all $\tilde{X},\tilde{Y},\tilde{Z}\in \Im _{0}^{1}(T_{1}^{1}(M))$. For
pairs $\tilde{X}=^{H}X,^{V}A$, $\widetilde{Y}=^{H}Y,^{V}B$ and $\widetilde{Z}%
=^{H}Z,{}^{V}C$, then we get%
\begin{eqnarray}
(\phi _{J}{}^{S}g_{f})({}^{H}X,{}^{V}B,{}^{H}Z) &=&2^{S}g_{f}({}^{V}B,(%
\widetilde{\gamma }-\gamma )R(X,Z))  \label{B3.9} \\
(\phi _{J}{}^{S}g_{f})({}^{H}X,{}^{H}Y,{}^{V}C) &=&2^{S}g_{f}((\widetilde{%
\gamma }-\gamma )R(X,Y),^{V}C).  \notag
\end{eqnarray}%
for all $X,Y,Z\in \Im _{0}^{1}(M)$ and $A,B,C\in \Im _{1}^{1}(M),$ and the
others is zero. Since $\phi _{J}{}^{S}g_{f}=0$ is equivalent to ${}%
\widetilde{{\nabla }}J=0$, we have the following theorem.

\begin{theorem}
\label{Theorem3}Let $(M,g)$ be a Riemannian manifold and let $T_{1}^{1}(M)$
be its tensor bundle equipped with the rescaled Sasaki type metric $%
^{S}g_{f} $ and the paracomplex structure $J$. The triple $\left(
T_{1}^{1}(M),J,{}^{S}g_{f}\right) $ is a para-K\"{a}hler-Norden
(paraholomorphic Norden) manifold if and only if $M$ is flat.
\end{theorem}

\begin{remark}
Let $(M,g)$ be a Riemannian manifold and let $T_{1}^{1}(M)$ be its tensor
bundle equipped with the rescaled Sasaki type metric $^{S}g_{f}.$ The
diagonal lift $^{D}\gamma $ of $\gamma \in \Im _{1}^{1}(M)$ to $T_{1}^{1}(M)$
is defined by the formulas%
\begin{eqnarray*}
^{D}\gamma ^{H}X &=&^{H}(\gamma (X)) \\
^{D}\gamma ^{V}\omega &=&-^{V}(\gamma (A))
\end{eqnarray*}%
for any $X\in \Im _{0}^{1}\left( M\right) $ and $A\in \Im _{1}^{1}\left(
M\right) $ \cite{Gezer0}. The diagonal lift $^{D}I$ of the identity tensor
field $I\in \Im _{1}^{1}(M)$ has the following properties%
\begin{eqnarray*}
^{D}I^{H}X &=&^{H}X \\
^{D}I^{V}\omega &=&-^{V}A
\end{eqnarray*}%
and satisfies $(^{D}I)^{2}=I_{T_{1}^{1}(M)}$. Thus, $^{D}I$ is an almost
paracomplex structure. Also, the rescaled Sasaki type metric $^{S}g_{f}$ is
pure with respect to $^{D}I$, i.e. the triple $\left(
T_{1}^{1}(M),^{D}I,{}^{S}g_{f}\right) $ is an almost paracomplex Norden
manifold. Finally, by using $\phi -$operator, we can say that the rescaled
Sasaki type metric $^{S}g_{f}$ is paraholomorphic with respect to $^{D}I$ if
and only if $M$ is flat.
\end{remark}

As is known that the almost paracomplex Norden structure is a specialized
Riemannian almost product structure on a Riemannian manifold. The theory of
Riemannian almost product structures was initiated by K. Yano in \cite{Yano}%
. The classification of Riemannian almost product structure with respect to
their covariant derivatives is described by A. M. Naveira in \cite{Naveira}.
This is the analogue of the classification of almost Hermitian structures by
A. Gray and L. Hervella in \cite{Gray}. Having in mind these results, M.
Staikova and K. Gribachev obtained a classification of the Riemannian almost
product structures, for which the trace vanishes (see \cite{Staikova}).
There are lots of physical applications involving a Riemannian almost
product manifold. Now we shall give some applications for almost paracomplex
Norden structures in context of almost product Riemannian manifolds.

\textit{4.1. }Let us recall almost product Riemannian manifolds. If an $n-$%
dimensional Riemannian manifold $M$, endowed with a positive definite
Riemannian metric $g$, admits a non-trivial tensor field $F$ of type $(1.1)$
such that 
\begin{equation*}
F^{2}=I
\end{equation*}%
and 
\begin{equation*}
g(FX,Y)=g(X,FY)
\end{equation*}%
for all $X,Y\in \Im _{0}^{1}(M)$, then $F$ is called an almost product
structure and $(M,F,g)$ is called an almost product Riemannian manifold. An
integrable almost product Riemannian manifold with structure tensor $F$ is
called a locally product Riemannian manifold. If $F$ is covariantly constant
with respect to the Levi-Civita connection $\nabla $ of $g$ which is
equivalent to $\phi _{F}g=0$, then ($M,F,g)$ is called a locally
decomposable Riemannian manifold.

Now consider the almost product structure ${J}$ defined by (\ref{B3.8}) and
the Levi-Civita connection $\widetilde{\nabla }$ given by Proposition \ref%
{Proposition1}. We define a $(1,2)$ tensor field on $T_{1}^{1}(M)$ by%
\begin{equation*}
\widetilde{S}(\widetilde{X},\widetilde{Y})=\frac{1}{2}\{(\widetilde{\nabla }%
_{J\widetilde{Y}}J)\widetilde{X}+J((\widetilde{\nabla }_{\widetilde{Y}}J)%
\widetilde{X})-J((\widetilde{\nabla }_{\widetilde{X}}J)\widetilde{Y})\}
\end{equation*}%
for all $\widetilde{X},\widetilde{Y}\in \Im _{0}^{1}(T_{1}^{1}(M))$. Then
the linear connection 
\begin{equation*}
^{(P)}\widetilde{\nabla }_{\widetilde{X}}\widetilde{Y}=\widetilde{\nabla }_{%
\widetilde{X}}\widetilde{Y}-\widetilde{S}(\widetilde{X},\widetilde{Y})
\end{equation*}%
is an almost product connection on $T_{1}^{1}(M)$ (for almost product
connection, see \cite{Leon}).

\begin{theorem}
Let $(M,g)$ be a Riemannian manifold and let $T_{1}^{1}(M)$ be its tensor
bundle equipped with the rescaled Sasaki type metric $^{S}g_{f}$ and the
almost product structure $J$. Then the almost product connection $^{(P)}%
\widetilde{\nabla }$ constructed by the Levi-Civita connection $\widetilde{%
\nabla }$ of $^{S}g_{f}$ and the almost product structure $J$ is as follows:%
\begin{equation}
\begin{array}{l}
{i)\mathrm{\;}}^{(P)}\widetilde{{\nabla }}{_{{}^{H}X}{}^{H}Y={}^{H}\left(
\nabla _{X}Y+\frac{1}{2f}^{f}A(X,Y)\right) ,} \\ 
{ii)\mathrm{\;}^{(P)}\widetilde{{\nabla }}_{{}^{H}X}{}^{V}B={}^{V}\left(
\nabla _{X}B\right) ,} \\ 
{iii)\mathrm{\;}^{(P)}\widetilde{{\nabla }}_{{}^{V}C}{}^{H}Y={\frac{3}{2f}}%
{}^{H}\left( g^{bl}\,R(t_{b},C_{l})Y+g_{at}(t^{a}\,(g^{-1}\circ R(\quad ,Y)%
\widetilde{C}^{\,t})\right) ,} \\ 
{iv)\mathrm{\;}^{(P)}\widetilde{{\nabla }}_{{}^{V}C}{}^{V}B=0}%
\end{array}
\label{B3.10}
\end{equation}
\end{theorem}

Similarly, by means of the almost product structure $^{D}I$ and the
Levi-Civita connection $\widetilde{\nabla }$ of the rescaled Sasaki type
metric $^{S}g_{f}$, another almost product connection can be constructed.
Denoting by $^{^{(P)}\widetilde{\nabla }}T$, the torsion tensor of $^{(P)}%
\widetilde{\nabla }$, we have from (\ref{A2.3}) and (\ref{B3.10})%
\begin{eqnarray*}
^{^{(P)}\widetilde{\nabla }}T(^{V}C,^{V}B) &=&0, \\
^{^{(P)}\widetilde{\nabla }}T(^{V}C,^{H}Y) &=&{{\frac{3}{2f}}{}^{H}\left(
g^{bl}\,R(t_{b},C_{l})Y+g_{at}(t^{a}\,(g^{-1}\circ R(\quad ,Y)\widetilde{C}%
^{\,t})\right) }, \\
^{^{(P)}\widetilde{\nabla }}T(^{H}X,^{H}Y) &=&-{(\tilde{\gamma}-\gamma
)R(X,Y)}.
\end{eqnarray*}%
Hence we have the theorem below.

\begin{theorem}
Let $(M,g)$ be a Riemannian manifold and let $T_{1}^{1}(M)$ be its tensor
bundle. The almost product connection $^{(P)}\widetilde{\nabla }$
constructed by the Levi-Civita connection $\widetilde{\nabla }$ of the
rescaled Sasaki type metric $^{S}g_{f}$ and the almost product structure $J$
is symmetric if and only if $M$ is flat.
\end{theorem}

As is well-known, if there exists a symmetric almost product connection on $%
M $ then the almost product structure $J$ is integrable \cite{Leon}. The
converse is also true \cite{Fujimoto}. If $J$ is covariantly constant with
respect to the Levi-Civita connection $\widetilde{\nabla }$ of the rescaled
Sasaki type metric $^{S}g_{f}$ which is equivalent to $\phi _{J}^{\text{ }%
S}g_{f}=0$, then $(T_{1}^{1}(M),J,^{S}g_{f})$ is called a locally
decomposable Riemannian manifold. In view of Theorem \ref{Theorem3}, we have
the following conclusion.

\begin{corollary}
Let $(M,g)$ be a Riemannian manifold and $T_{1}^{1}(M)$ be its tensor bundle
equipped with the rescaled Sasaki type metric $^{S}g_{f}$ and the almost
product structure $J$. The triple $(T_{1}^{1}(M),J,^{S}g_{f})$ is a locally
decomposable Riemannian manifold if and only if $M$ is flat.
\end{corollary}

\textit{4.2.} Let $(M_{2k},\varphi ,g)$ be a non-integrable almost
paracomplex manifold with a Norden metric. An almost paracomplex Norden
manifold $(M_{2k},\varphi ,g)$ is a quasi-para-K\"{a}hler--Norden manifold,
if $\underset{X,Y,Z}{\sigma }g((\nabla _{X}\varphi )Y,Z)=0$, where $\sigma $
is the cyclic sum by three arguments \cite{Manev}. In \cite{Salimov4}, the
authors proved that $\underset{X,Y,Z}{\sigma }g((\nabla _{X}\varphi )Y,Z)=0$
is equivalent to $(\phi _{\varphi }g)(X,Y,Z)+(\phi _{\varphi
}g)(Y,Z,X)+(\phi _{\varphi }g)(Z,X,Y)=0.$ We compute%
\begin{equation*}
A(\tilde{X},\tilde{Y},\tilde{Z})=(\phi _{J}{}^{S}g_{f})(\tilde{X},\tilde{Y},%
\tilde{Z})+(\phi _{J}{}^{S}g_{f})(\widetilde{Y},\widetilde{Z},\widetilde{X}%
)+(\phi _{J}{}^{S}g_{f})(\widetilde{Z},\widetilde{X},\widetilde{Y})
\end{equation*}%
for all $\widetilde{X},\widetilde{Y},\widetilde{Z}\in \Im
_{0}^{1}(T_{1}^{1}(M)).$ By means of (\ref{B3.9}), we have $A(\tilde{X},%
\tilde{Y},\tilde{Z})=0$ for all $\widetilde{X},\widetilde{Y},\widetilde{Z}%
\in \Im _{0}^{1}(T_{1}^{1}(M)).$ Hence we state the following theorem.

\begin{theorem}
\label{Theorem4}Let $(M,g)$ be a Riemannian manifold and $T_{1}^{1}(M)$ be
its tensor bundle equipped with the rescaled Sasaki type metric $^{S}g_{f}$
and the almost paracomplex structure $J$ defined by (\ref{B3.8}). The triple 
$(T_{1}^{1}(M),J,^{S}g_{f})$ is a quasi-para-K\"{a}hler--Norden manifold.
\end{theorem}

O. Gil-Medrano and A. M. Naveira proved that both distributions of the
almost product structure on the Riemannian manifold ($M,\varphi ,g)$ are
totally geodesic if and only if $\underset{X,Y,Z}{\sigma }g((\nabla
_{X}\varphi )Y,Z)=0$ for any $X,Y,Z\in \Im _{0}^{1}(M)$ \cite{Gil}$.$ As a
consequence of Theorem \ref{Theorem4}, we have the following.

\begin{corollary}
Both distributions of the almost product Riemannian manifold $(T_{1}^{1}(M),$
$J,^{S}g_{f})$ are totally geodesic.
\end{corollary}

\textit{4.3. }The famous Golden section $\eta =\frac{1+\sqrt{5}}{2}\approx
1,61803398874989...$ \ being the root of the equation $x^{2}-x-1=0$ is an
irrational number which has many applications in mathematics, computational
science, biology, art, architecture, nature, etc. \noindent In the last few
years, the Golden proportion has played an increasing role in modern
physical research and it has a unique significant role in atomic physics 
\cite{Heyrovska}. The Golden proportion has also interesting properties in
topology of four-manifolds, in conformal field theory, in mathematical
probability theory and in Cantorian spacetime \cite{Marek1,Marek2}. Inspired
by Golden Ratio, a new structure on a Riemannian manifold which named the
golden structure was constructed by M. Crasmareanu and C. Hretcanu \cite%
{Cra,Hre2,Hre1}. For a manifold $M$, let $\psi $ be a $(1,1)$-tensor field
on $M$. If the polynomial $X^{2}-X-1$ is the minimal polynomial for a
structure $\varphi $ satisfying $\psi ^{2}-\psi -I=0$ , then $\psi $ is
called a golden structure on $M$ and $(M,\psi )$ is a golden manifold. Let $%
(M,g)$ be a Riemannian manifold endowed with the Golden structure $\psi $
such that 
\begin{equation*}
g(\psi X,Y)=g(X,\psi Y),
\end{equation*}%
for all $X,Y\in \Im _{0}^{1}(M)$. The triple $(M,\psi ,g)$ is named a Golden
Riemannian manifold.

If $\psi $ is a Golden structure on $M$, then 
\begin{equation}
F=\frac{1}{\sqrt{5}}(2\psi -I)  \label{B3.11}
\end{equation}%
is an almost product structure on $M$. Conversely, 
\begin{equation}
\psi =\frac{1}{2}(I+\sqrt{5}F)  \label{B3.12}
\end{equation}%
is a Golden structure on $M$. If a Riemannian metric $g$ is pure with
respect to an almost product structure $F$, then the Riemannian metric $g$
is pure with respect to the corresponding Golden structure $\psi .$ A simple
computation, using the expression of the corresponding almost product
structure via (\ref{B3.11}) gives: 
\begin{equation}
\phi _{F}g=\frac{2}{\sqrt{5}}\phi _{\psi }g.  \label{B3.13}
\end{equation}

In \cite{Gezer}, the first author and collaborators have proved that 1) Let $%
(M,\psi ,g)$ be a Golden Riemannian manifold and $F$ its corresponding
almost product structure. The golden structure $\psi $ is integrable if $%
\phi _{\psi }g=0$ (or equivalently $\phi _{F}g=0)$ and 2) Let $(M,\psi ,g)$
be a Golden Riemannian manifold and $F$ its corresponding almost product
structure. The manifold $M$ is a locally decomposable Golden Riemannian
manifold if and only if $\phi _{F}g=0$ (or equivalently $\phi _{\psi }g=0$ ).

By means of the almost product structure $J$, from (\ref{B3.12}) we can
construct a Golden structure on $T_{1}^{1}(M)$ defined by the formulas%
\begin{equation}
\left\{ 
\begin{array}{l}
\widetilde{\psi }{({}^{H}X)=(}\frac{1-\sqrt{5}}{2}){{}^{H}X,} \\ 
\widetilde{\psi }{({}^{V}A)={(}\frac{1+\sqrt{5}}{2})^{V}A.}%
\end{array}%
\right.  \label{B3.14}
\end{equation}%
for any $X\in \Im _{0}^{1}(M)$ and $A\in \Im _{1}^{1}(M)$. Also the
following hold%
\begin{equation*}
{}^{S}g_{f}\left( \widetilde{\psi }\widetilde{X},\widetilde{Y}\right)
={}{}^{S}g_{f}\left( \widetilde{X},\widetilde{\psi }\widetilde{Y}\right)
\end{equation*}%
for any $\widetilde{X},\widetilde{Y}\in \Im _{0}^{1}\left(
T_{1}^{1}(M)\right) $, i.e. ${}^{S}g_{f}$ is pure with respect to $%
\widetilde{\psi }$. In view of Theorem \ref{Theorem3}, by (\ref{B3.13}), we
have the following result.

\begin{corollary}
Let $(M,g)$ be a Riemannian manifold and let $T_{1}^{1}(M)$ be its tensor
bundle equipped with the rescaled Sasaki type metric $^{S}g_{f}$ and the
Golden structure $\widetilde{\psi }$ associated with the almost product
structure $J$. The triple $\left( T_{1}^{1}(M),\widetilde{\psi }%
,{}^{S}g_{f}\right) $ is a a locally decomposable Golden Riemannian manifold
if and only if $M$ is flat.
\end{corollary}

\begin{remark}
Another Golden structure associated with the almost product structure $^{D}I$
is as follows:%
\begin{equation*}
\left\{ 
\begin{array}{l}
\overline{\psi }{({}^{H}X)=(}\frac{1+\sqrt{5}}{2}){{}^{H}X,} \\ 
\overline{\psi }{({}^{V}A)={(}\frac{1-\sqrt{5}}{2})^{V}A.}%
\end{array}%
\right.
\end{equation*}%
Similarly, we say that the triple $\left( T_{1}^{1}(M),\overline{\psi }%
,{}^{S}g_{f}\right) $ is a a locally decomposable Golden Riemannian manifold
if and only if $M$ is flat.\bigskip
\end{remark}

\section{\protect\bigskip Curvature properties of the rescaled Sasaki type
metric}

The Riemannian curvature tensor $\widetilde{R}$ of $T_{1}^{1}(M)$ with the
rescaled Sasaki type metric\textit{\ ${}$}$^{S}g_{f}$ is obtained from the
well-known formula%
\begin{equation*}
\widetilde{R}\left( \widetilde{X},\widetilde{Y}\right) \widetilde{Z}=%
\widetilde{\nabla }_{\widetilde{X}}\widetilde{\nabla }_{\widetilde{Y}}%
\widetilde{Z}-\widetilde{\nabla }_{\widetilde{Y}}\widetilde{\nabla }_{%
\widetilde{X}}\widetilde{Z}-\widetilde{\nabla }_{\left[ \widetilde{X},%
\widetilde{Y}\right] }\widetilde{Z}
\end{equation*}%
for all $\widetilde{X},\widetilde{Y},\widetilde{Z}\in \Im
_{0}^{1}(T_{1}^{1}(M))$. Then from Lemma \ref{Lemma1} and Theorem \ref%
{Theorem1}, we get the following proposition.

\begin{proposition}
\label{Proposition2}The components of the curvature tensor $\widetilde{R}$
of the tensor bundle $T_{1}^{1}(M)$ with the rescaled Sasaki type metric%
\textit{\ ${}$}$^{S}g_{f}$ are given as follows:%
\begin{eqnarray}
&&{\widetilde{R}(E_{m},E_{l})E_{j}}  \label{B3.15} \\
&{=}&\left\{ {R_{mlj}^{\text{ \ \ \ }r}}\right.  \notag \\
&&{+\frac{1}{4f}(g_{ka}R_{.}^{s}{}_{.}^{h}{}_{m}^{\text{ \ }%
r}R_{ljh}{}^{p}-g_{ka}R_{.}^{s}{}_{.}^{h}{}_{l}^{\text{ }r}R_{mjh}^{\text{ \
\ \ }p}-2g_{ka}R_{.}^{s}{}_{.}^{h}{}_{j}^{\text{ }r}R_{mlh}^{\text{ \ \ \ }%
p})t_{s}^{a}t_{p}^{k}}  \notag \\
&&{+\frac{1}{4f}(g_{ka}R_{.}^{s}{}_{.}^{h}{}_{l}^{\text{ }r}R_{mjp}^{\text{
\ \ \ }k}-g_{ka}R_{.}^{s}{}_{.}^{h}{}_{m}^{\text{ \ }r}R_{ljp}^{\text{ \ \ }%
k}+2g_{ka}R_{.}^{s}{}_{.}^{h}{}_{j}^{\text{ \ }r}R_{mlp}^{\text{ \ \ \ }%
k})t_{s}^{a}t_{h}^{p}}  \notag \\
&&{+\frac{1}{4f}(g^{hb}R_{kpl}^{\text{ \ \ \ }r}R_{mjh}^{\text{ \ \ \ }%
s}-g^{hb}R_{kpm}^{\text{ \ \ \ }r}R_{ljh}^{\text{ \ \ \ }s}+2g^{hb}R_{kpj}^{%
\text{ \ \ \ }r}R_{mlh}^{\text{ \ \ \ }s})t_{b}^{p}t_{s}^{k}}  \notag \\
&&{+\frac{1}{4f}(g^{hb}R_{ksm}^{\text{ \ \ \ \ }r}R_{ljp}^{\text{ \ \ \ }%
k}-g^{hb}R_{ksl}^{\text{ \ \ \ }r}R_{mjp}^{\text{ \ \ \ }k}-2g^{hb}R_{ksj}^{%
\text{ \ \ }r}R_{mlp}^{\text{ \ \ \ }k})t_{b}^{s}t_{h}^{p}}  \notag \\
&&{+\nabla _{m}(\frac{1}{2f}A_{lj}^{r})-\nabla _{l}(\frac{1}{2f}A_{mj}^{r})+%
\frac{1}{4f^{2}}A_{ms}^{r}A_{lj}^{s}-\frac{1}{4f^{2}}A_{ls}^{r}A_{mj}^{s}%
\}E_{r}}  \notag \\
&&{+\{\frac{1}{2}(\nabla _{m}R_{ljr}^{\text{ \ \ \ }s}-\nabla _{l}R_{mjr}^{%
\text{ \ \ \ \ }s})t_{s}^{v}+\frac{1}{2}(\nabla _{l}R_{mjs}^{\text{ \ \ \ }%
v}-\nabla _{m}R_{ljs}^{\text{ \ \ }v})t_{r}^{s}}  \notag \\
&&{+\frac{1}{4f}((R_{mhr}^{\text{ \ \ \ }s}t_{s}^{v}-R_{mhs}^{\text{ \ \ \ }%
v}t_{r}^{s})A_{lj}^{h}-(R_{lhr}^{\text{ \ \ \ }s}t_{s}^{v}-R_{lhs}^{\text{ \
\ \ }v}t_{r}^{s})A_{mj}^{h})\}E_{\overline{r}},}  \notag
\end{eqnarray}%
\begin{eqnarray*}
{\widetilde{R}(E_{\overline{m}},E_{l})E_{j}} &{=}&\left\{ {-\frac{1}{2f}%
g_{na}(\nabla _{l}R_{.}^{s}{}_{.}^{m}{}_{j}^{\text{ }r})t_{s}^{a}+\frac{1}{2f%
}g^{mb}(\nabla _{l}R_{nsj}^{\text{ \ \ \ }r})t_{b}^{s}}\right. \\
&&{+\frac{1}{4f^{2}}(g_{na}R_{.}^{s}{}_{.}^{m}{}_{h}^{\text{ }%
r}A_{lj}^{h}t_{s}^{a}-g^{mb}R_{nsh}^{\text{ \ \ \ }%
r}A_{lj}^{h}t_{b}^{s}+g^{mb}R_{nsj}^{\text{ \ \ \ }h}A_{lh}^{r}t_{b}^{s}} \\
&&\left. {-g_{na}R_{.}^{s}{}_{.}^{m}{}_{j}^{\text{ }%
h}A_{lh}^{r}t_{s}^{a}+2f_{l}g_{na}R_{.}^{s}{}_{.}^{m}{}_{j}^{\text{ }%
r}t_{s}^{a}-2f_{l}g^{mb}R_{nsj}^{\text{ \ \ \ }r}t_{b}^{s})}\right\} {E_{r}}
\\
&&{+\{\frac{1}{2}R_{ljr}^{\text{ \ \ \ }m}\delta _{n}^{v}-\frac{1}{2}%
R_{ljn}^{\text{ \ \ \ }v}\delta _{r}^{m}-\frac{1}{4f}(R_{lhr}^{\text{ \ \ \ }%
s}g_{na}R_{.}^{p}{}_{.}^{m}{}_{j}^{\text{ }h})t_{s}^{v}t_{p}^{a}} \\
&&{+\frac{1}{4f}(R_{lhr}^{\text{ \ \ \ }s}g^{mb}R_{npj}^{\text{ \ \ \ }%
h})t_{s}^{v}t_{b}^{p}+\frac{1}{4f}(R_{lhs}^{\text{ \ \ \ }%
v}g_{na}R_{.}^{p}{}_{.}^{m}{}_{j}^{\text{ }h})t_{r}^{s}t_{p}^{a}} \\
&&{-\frac{1}{4f}(R_{lhs}^{\text{ \ \ \ }v}g^{mb}R_{npj}^{\text{ \ \ \ }%
h})t_{r}^{s}t_{b}^{p}\}E_{\overline{r}},} \\
{\widetilde{R}(E_{m},E_{\overline{l}})E_{j}} &{=}&\left\{ {\frac{1}{2f}%
g_{ta}(\nabla _{m}R_{.}^{s}{}_{.}^{l}{}_{j}^{\text{ }r})t_{s}^{a}-\frac{1}{2f%
}g^{lb}(\nabla _{m}R_{tsj}^{\text{ \ \ \ }r})t_{b}^{s}}\right. \\
&&{+\frac{1}{4f^{2}}(g_{ta}R_{.}^{s}{}_{.}^{l}{}_{j}^{\text{ }%
h}A_{mh}^{r}t_{s}^{a}-g^{lb}R_{tsj}^{\text{ \ \ \ }%
h}A_{mh}^{r}t_{b}^{s}+g^{lb}R_{tsh}^{\text{ \ \ \ }r}A_{mj}^{h}t_{b}^{s}} \\
&&{-g_{ta}R_{.}^{s}{}_{.}^{l}{}_{h}^{\text{ }%
r}A_{mj}^{h}t_{s}^{a}-2f_{m}g_{ta}R_{.}^{s}{}_{.}^{l}{}_{j}^{\text{ }%
h}t_{s}^{a}+2f_{m}g^{lb}R_{tsj}^{\text{ \ \ \ }r}t_{b}^{s})\}E_{r}} \\
&&{+\{-\frac{1}{2}R_{mjr}^{\text{ \ \ \ }l}\delta _{t}^{v}+\frac{1}{2}%
R_{mjt}^{\text{ \ \ \ }v}\delta _{r}^{l}+\frac{1}{4f}(R_{mhr}^{\text{ \ \ \ }%
s}g_{va}R_{.}^{p}{}_{.}^{l}{}_{j}^{\text{ }h})t_{s}^{v}t_{p}^{a}} \\
&&{-\frac{1}{4f}(R_{mhr}^{\text{ \ \ \ }s}g^{lb}R_{tpj}^{\text{ \ \ \ }%
h})t_{s}^{v}t_{b}^{p}-\frac{1}{4f}(R_{mhp}^{\text{ \ \ \ }%
v}g_{ta}R_{.}^{s}{}_{.}^{l}{}_{j}^{\text{ }h})t_{r}^{p}t_{s}^{a}} \\
&&{+\frac{1}{4f}(R_{mhs}^{\text{ \ \ \ }v}g^{lb}R_{tpj}^{\text{ \ \ \ }%
h})t_{r}^{s}t_{b}^{p}\}E_{\overline{r}},} \\
{\widetilde{R}(E_{\overline{m}},E_{\overline{l}})E_{j}} &{=}&\left\{ {\frac{1%
}{f}g_{tn}R_{.}^{m}{}_{.}^{l}{}_{j}^{\text{ }r}-\frac{1}{f}g^{lm}R_{tnj}^{%
\text{ \ \ \ }r}}\right. \\
&&{+\frac{1}{4f^{2}}(g_{na}R_{.}^{s}{}_{.}^{m}{}_{h}^{\text{ }%
r}g_{tb}R_{.}^{p}{}_{.}^{l}{}_{j}^{\text{ }%
h}-g_{ta}R_{.}^{s}{}_{.}^{l}{}_{h}^{\text{ }%
r}g_{nb}R_{.}^{p}{}_{.}^{m}{}_{j}^{\text{ }h})t_{s}^{a}t_{p}^{b}} \\
&&{+\frac{1}{4f^{2}}(g_{ta}R_{.}^{s}{}_{.h}^{l\text{ }}{}^{r}g^{mb}R_{npj}^{%
\text{ \ \ \ }h}-g_{na}R_{.}^{s}{}_{.}^{m}{}_{h}^{\text{ }r}g^{lb}R_{tpj}^{%
\text{ \ \ \ }h})t_{s}^{a}t_{b}^{p}} \\
&&{+\frac{1}{4f^{2}}(g^{lb}R_{tph}^{\text{ \ \ \ }%
r}g_{na}R_{.}^{s}{}_{.}^{m}{}_{j}^{\text{ }h}-g^{mb}R_{nph}^{\text{ \ \ \ }%
r}g_{ta}R_{.}^{s}{}_{.}^{l}{}_{j}^{\text{ }h})t_{b}^{p}t_{s}^{a}} \\
&&{+\frac{1}{4f^{2}}(g^{ma}R_{nsh}^{\text{ \ \ \ }r}g^{lb}R_{tsj}^{\text{ \
\ \ }h}-g^{la}R_{tsh}^{\text{ \ \ \ }r}g^{mb}R_{npj}^{\text{ \ \ \ }%
h})t_{a}^{s}t_{b}^{p}\}E_{r},} \\
{\widetilde{R}(E_{m},E_{l})E_{\overline{j}}} &{=}&\left\{ {\frac{1}{2f}%
g_{ia}(\nabla _{m}R_{.}^{s}{}_{.}^{j}{}_{l}^{\text{ }r}-\nabla
_{l}R_{.}^{s}{}_{.}^{j}{}_{m}^{\text{ }r})t_{s}^{a}+\frac{1}{2f}%
g^{jb}(\nabla _{l}R_{ism}^{\text{ \ \ \ }r}-\nabla _{m}R_{isl}^{\text{ \ \ \ 
}r})t_{b}^{s}}\right. \\
&&{+\frac{1}{4f^{2}}(g_{ia}R_{.}^{s}{}_{.}^{j}{}_{l}^{\text{ }%
h}A_{mh}^{r}t_{s}^{a}-g^{jb}R_{isl}^{\text{ \ \ \ }%
h}A_{mh}^{r}t_{b}^{s}+g^{jb}R_{ism}^{\text{ \ \ \ }%
h}A_{lh}^{r}t_{b}^{s}-g_{ia}R_{.}^{s}{}_{.}^{j}{}_{m}^{\text{ }%
h}A_{lh}^{r}t_{s}^{a}} \\
&&{-2(f_{m}g_{ia}R_{.}^{s}{}_{.}^{j}{}_{l}^{\text{ }%
h}t_{s}^{a}+f_{m}g^{jb}R_{isl}^{\text{ \ \ \ }%
r}t_{b}^{s}+f_{l}g_{ia}R_{.}^{s}{}_{.}^{j}{}_{m}^{\text{ }%
r}t_{s}^{a}-f_{l}g^{jb}R_{ism}^{\text{ \ \ \ }r}t_{b}^{s})\}E_{r}} \\
&&{+\{R_{mli}^{\text{ \ \ \ }v}\delta _{r}^{j}-R_{mlr}^{\text{ \ \ \ }%
j}\delta _{i}^{v}+\frac{1}{4f}(R_{mhr}^{\text{ \ \ \ }%
s}g_{ia}R_{.}^{p}{}_{.}^{j}{}_{l}^{\text{ }h}-R_{lhr}^{\text{ \ \ \ }%
s}g_{ia}R_{.}^{p}{}_{.}^{j}{}_{m}^{\text{ }h})t_{s}^{v}t_{p}^{a}} \\
&&{+\frac{1}{4f}(R_{lhr}^{\text{ \ \ \ }s}g^{jb}R_{lpm}^{\text{ \ \ \ }%
h}-R_{mhr}^{\text{ \ \ \ }s}g^{jb}R_{ipl}^{\text{ \ \ \ }%
h})t_{s}^{v}t_{b}^{p}+\frac{1}{4f}(R_{lhp}^{\text{ \ \ \ }%
v}g_{ia}R_{.}^{s}{}_{.}^{j}{}_{m}^{\text{ }h}} \\
&&{-R_{mhp}^{\text{ \ \ \ }v}g_{ia}R_{.}^{s}{}_{.}^{j}{}_{l}^{\text{ }%
h})t_{r}^{p}t_{s}^{a}+\frac{1}{4f}(R_{mhs}^{\text{ \ \ \ }v}g^{jb}R_{ipl}^{%
\text{ \ \ \ }h}-R_{lhs}^{\text{ \ \ \ }v}g^{jb}R_{ipm}^{\text{ \ \ \ }%
h})t_{r}^{s}t_{b}^{p}\}E_{\overline{r}},} \\
{\widetilde{R}(E_{m},E_{\overline{l}})E_{\overline{j}}} &{=}&\left\{ {-\frac{%
1}{2f}g_{it}R_{.}^{l}{}_{.}^{j}{}_{m}^{\text{ }r}+\frac{1}{2f}g^{jl}R_{itm}^{%
\text{ \ \ \ }r}-\frac{1}{4f^{2}}(g_{ta}R_{.}^{s}{}_{.}^{l}{}_{h}^{\text{ }%
r}g_{ib}R_{.}^{p}{}_{.}^{l}{}_{m}^{\text{ }h})t_{s}^{a}t_{p}^{b}}\right. \\
&&{+\frac{1}{4f^{2}}(g_{ta}R_{.}^{s}{}_{.}^{l}{}_{h}^{\text{ }%
r}g^{jb}R_{ipm}^{\text{ \ \ \ }h})t_{s}^{a}t_{b}^{p}+\frac{1}{4f^{2}}%
(g^{lb}R_{tph}^{\text{ \ \ \ }r}g_{ia}R_{.}^{s}{}_{.}^{j}{}_{m}^{\text{ }%
h})t_{b}^{p}t_{s}^{a}} \\
&&{-\frac{1}{4f^{2}}(g^{la}R_{tsh}^{\text{ \ \ \ }r}g^{jb}R_{ipm}^{\text{ \
\ \ }h})t_{a}^{s}t_{b}^{p}\}E_{r},}
\end{eqnarray*}

\begin{eqnarray*}
{\widetilde{R}(E_{\overline{m}},E_{l})E_{\overline{j}}} &{=}&\left\{ {\frac{1%
}{2f}g_{in}R_{.}^{m}{}_{.}^{j}{}_{l}^{\text{ }r}-\frac{1}{2f}g^{jm}R_{inl}^{%
\text{ \ \ \ }r}+\frac{1}{4f^{2}}(g_{na}R_{.}^{s}{}_{.}^{m}{}_{h}^{\text{ }%
r}g_{ib}R_{.}^{p}{}_{.}^{j}{}_{l}^{\text{ }h})t_{s}^{a}t_{p}^{b}}\right. \\
&&{-\frac{1}{4f^{2}}(g_{na}R_{.}^{s}{}_{.}^{m}{}_{h}^{\text{ }%
r}g^{jb}R_{ipl}^{\text{ \ \ \ }h})t_{s}^{a}t_{b}^{p}-\frac{1}{4f^{2}}%
(g^{mb}R_{nph}^{\text{ \ \ \ }r}g_{ia}R_{.}^{s}{}_{.}^{j}{}_{m}^{\text{ }%
h})t_{b}^{p}t_{s}^{a}} \\
&&{+\frac{1}{4f^{2}}(g^{ma}R_{nsh}^{\text{ \ \ \ }r}g^{jb}R_{ipl}^{\text{ \
\ \ }h})t_{a}^{s}t_{b}^{p}\}E_{r}} \\
{\widetilde{R}(E_{\overline{m}},E_{\overline{l}})E_{\overline{j}}} &{=}&{0}
\end{eqnarray*}%
with respect to the adapted frame $\left\{ E_{\alpha }\right\} $ (for $f=1,$
see \cite{Salimov3})$.$
\end{proposition}

We now compare the geometries of the Riemannian manifold $(M,g)$ and its
tensor bundle $T_{1}^{1}(M)$ with the rescaled Sasaki type metric $%
{}^{S}g_{f}$ $.$

\begin{theorem}
Let $(M,g)$ be a Riemannian manifold and $T_{1}^{1}(M)$ be its tensor bundle
with the rescaled Sasaki type metric ${}^{S}g_{f}$. Then $T_{1}^{1}(M)$ is
locally flat if and only if $M$ is locally flat and ${\nabla _{m}(\frac{1}{2f%
}A_{lj}^{r})-\nabla _{l}(\frac{1}{2f}A_{mj}^{r})+\frac{1}{4f^{2}}%
A_{ms}^{r}A_{lj}^{s}-\frac{1}{4f^{2}}A_{ls}^{r}A_{mj}^{s}}=0{}.$
\end{theorem}

\begin{proof}
If 
\begin{equation*}
{}{\nabla _{m}(\frac{1}{2f}A_{lj}^{r})-\nabla _{l}(\frac{1}{2f}A_{mj}^{r})+%
\frac{1}{4f^{2}}A_{ms}^{r}A_{lj}^{s}-\frac{1}{4f^{2}}A_{ls}^{r}A_{mj}^{s}}=0
\end{equation*}%
in the equations (\ref{B3.15}), then $R\equiv 0$ implies ${\widetilde{R}}%
\equiv 0$. Conversely, if we assume ${\widetilde{R}}\equiv 0,$ then from the
first equation in (\ref{B3.15}) in the point $(x^{i},t_{i}^{j})=(x^{i},0)\in
T_{1}^{1}(M),$ we get%
\begin{eqnarray*}
&&({\,{}\,\widetilde{R}(E_{m},E_{l})E_{j})}_{(x^{i},0)} \\
&{=}&{[}\left\{ {R_{mlj}^{\text{ \ \ \ }r}}\right. \\
&&{+\frac{1}{4f}(g_{ka}R_{.}^{s}{}_{.}^{h}{}_{m}^{\text{ \ }r}R_{ljh}^{\text{
\ \ \ }p}-g_{ka}R_{.}^{s}{}_{.}^{h}{}_{l}^{\text{ }r}R_{mjh}^{\text{ \ \ \ }%
p}-2g_{ka}R_{.}^{s}{}_{.}^{h}{}_{j}^{\text{ }r}R_{mlh}^{\text{ \ \ \ }%
p})t_{s}^{a}t_{p}^{k}} \\
&&{+\frac{1}{4f}(g_{ka}R_{.}^{s}{}_{.}^{h}{}_{l}^{\text{ }r}R_{mjp}^{\text{
\ \ \ }k}-g_{ka}R_{.}^{s}{}_{.}^{h}{}_{m}^{\text{ \ }r}R_{ljp}^{\text{ \ \ \ 
}k}+2g_{ka}R_{.}^{s}{}_{.}^{h}{}_{j}^{\text{ \ }r}R_{mlp}^{\text{ \ \ \ }%
k})t_{s}^{a}t_{h}^{p}} \\
&&{+\frac{1}{4f}(g^{hb}R_{kpl}^{\text{ \ \ \ }r}R_{mjh}^{\text{ \ \ \ }%
s}-g^{hb}R_{kpm}^{\text{ \ \ \ }r}R_{ljh}^{\text{ \ \ \ }s}+2g^{hb}R_{kpj}^{%
\text{ \ \ \ }r}R_{mlh}^{\text{ \ \ \ }s})t_{b}^{p}t_{s}^{k}} \\
&&{+\frac{1}{4f}(g^{hb}R_{ksm}^{\text{ \ \ \ }r}R_{ljp}^{\text{ \ \ \ }%
k}-g^{hb}R_{ksl}^{\text{ \ \ \ }r}R_{mjp}^{\text{ \ \ \ }k}-2g^{hb}R_{ksj}^{%
\text{ \ \ \ }r}R_{mlp}^{\text{ \ \ \ }k})t_{b}^{s}t_{h}^{p}} \\
&&{+\nabla _{m}(\frac{1}{2f}A_{lj}^{r})-\nabla _{l}(\frac{1}{2f}A_{mj}^{r})+%
\frac{1}{4f^{2}}A_{ms}^{r}A_{lj}^{s}-\frac{1}{4f^{2}}A_{ls}^{r}A_{mj}^{s}%
\}E_{r}} \\
&&{+\{\frac{1}{2}(\nabla _{m}R_{ljr}^{\text{ \ \ \ }s}-\nabla _{l}R_{mjr}^{%
\text{ \ \ \ }s})t_{s}^{v}+\frac{1}{2}(\nabla _{l}R_{mjs}^{\text{ \ \ \ }%
v}-\nabla _{m}R_{ljs}^{\text{ \ \ }v})t_{r}^{s}} \\
&&{+\frac{1}{4f}((R_{mhr}^{\text{ \ \ \ }s}t_{s}^{v}-R_{mhs}^{\text{ \ \ \ }%
v}t_{r}^{s})A_{lj}^{h}-(R_{lhr}^{\text{ \ \ \ }s}t_{s}^{v}-R_{lhs}^{\text{ \
\ \ }v}t_{r}^{s})A_{mj}^{h})\}E_{\overline{r}}]}_{(x^{i},0)} \\
&=&{R_{mlj}^{\text{ \ \ \ }r}+\nabla _{m}(\frac{1}{2f}A_{lj}^{r})-\nabla
_{l}(\frac{1}{2f}A_{mj}^{r})+\frac{1}{4f^{2}}A_{ms}^{r}A_{lj}^{s}-\frac{1}{%
4f^{2}}A_{ls}^{r}A_{mj}^{s}} \\
&=&0
\end{eqnarray*}%
then $R\equiv 0$ and ${}{\nabla _{m}(\frac{1}{2f}A_{lj}^{r})-\nabla _{l}(%
\frac{1}{2f}A_{mj}^{r})+\frac{1}{4f^{2}}A_{ms}^{r}A_{lj}^{s}-\frac{1}{4f^{2}}%
A_{ls}^{r}A_{mj}^{s}}=0$. This completes the proof.
\end{proof}

\begin{corollary}
Let $(M,g)$ be a Riemannian manifold and $T_{1}^{1}(M)$ be its tensor bundle
with the rescaled Sasaki type metric ${}^{S}g_{f}$. Suppose that $%
f=C(const.) $, then $(T_{1}^{1}(M),^{S}g_{f})$ is locally flat if and only
if M is locally flat.
\end{corollary}

\bigskip We now turn our attention to the Ricci tensor and scalar curvature
of the rescaled Sasaki type metric ${}^{S}g_{f}$. Let $\widetilde{R}_{\alpha
\beta }=$ $\widetilde{R}_{\sigma \alpha \beta }^{\text{ \ \ \ \ \ \ }\sigma
} $ and $\widetilde{r}=(^{S}g_{f})^{\alpha \beta }$ $\widetilde{R}_{\alpha
\beta }$ denote the Ricci tensor and scalar curvature of the rescaled Sasaki
type metric ${}^{S}g_{f}$, respectively. From (\ref{B3.15}), the components
of the Ricci tensor $\widetilde{R}_{\alpha \beta }$ are characterized by%
\begin{eqnarray}
\widetilde{{R}}{_{\overline{l}\overline{j}}} &{=}&  \label{B3.16} \\
&&{-\frac{1}{4f^{2}}(g_{ta}R_{.}^{s}{}_{.}^{l}{}_{h}^{\text{ }%
r}g_{ib}R_{.}^{p}{}_{.}^{j}{}_{r}^{\text{ }h})t_{s}^{a}t_{p}^{b}+\frac{1}{%
4f^{2}}(g_{ta}R_{.}^{s}{}_{.}^{l}{}_{h}^{\text{ }r}g^{jb}R_{ipr}^{\text{ \ \
\ }h})t_{s}^{a}t_{b}^{p}}  \notag \\
&&{+\frac{1}{4f^{2}}(g^{lb}R_{tph}^{\text{ \ \ \ }%
r}g_{ia}R_{.}^{s}{}_{.}^{j}{}_{r}^{\text{ }h})t_{b}^{p}t_{s}^{a}-\frac{1}{%
4f^{2}}(g^{lb}R_{tsh}^{\text{ \ \ \ }r}g^{ja}R_{ipr}^{\text{ \ \ \ }%
h})t_{b}^{s}t_{a}^{p},}  \notag \\
\widetilde{{R}}{_{\overline{l}j}} &{=}&  \notag \\
&&{\frac{1}{2f}g_{ta}(\nabla _{r}R_{.}^{s}{}_{.}^{l}{}_{j}^{\text{ }%
r})t_{s}^{a}-\frac{1}{2f}g^{lb}(\nabla _{r}R_{tsj}^{\text{ \ \ \ }%
r})t_{b}^{s}}  \notag \\
&&{+\frac{1}{4f^{2}}(g_{ta}R_{.}^{s}{}_{.}^{l}{}_{j}^{\text{ }%
h}A_{rh}^{r}t_{s}^{a}-g^{lb}R_{tsj}^{\text{ \ \ \ }%
h}A_{rh}^{r}t_{b}^{s}+g^{lb}R_{tsh}^{\text{ \ \ \ }r}A_{rj}^{h}t_{b}^{s}} 
\notag \\
&&{-g_{ta}R_{.}^{s}{}_{.}^{l}{}_{h}^{\text{ }%
r}A_{rj}^{h}t_{s}^{a}-2f_{r}g_{ta}R_{.}^{s}{}_{.}^{l}{}_{j}^{\text{ }%
h}t_{s}^{a}+2f_{r}g^{lb}R_{tsj}^{\text{ \ \ \ }r}t_{b}^{s}),}  \notag \\
\widetilde{{R}}{_{l\overline{j}}} &{=}&  \notag \\
&&{\frac{1}{2f}g_{ia}(\nabla _{r}R_{.}^{s}{}_{.}^{j}{}_{l}^{\text{ }%
r})t_{s}^{a}-\frac{1}{2f}g^{jb}(\nabla _{r}R_{isl}^{\text{ \ \ \ }%
r})t_{b}^{s}}  \notag \\
&&{+\frac{1}{4f^{2}}(g_{ia}R_{.}^{s}{}_{.}^{j}{}_{l}^{\text{ }%
h}A_{rh}^{r}t_{s}^{a}-g^{jb}R_{isl}^{\text{ \ \ \ }h}A_{rh}^{r}t_{b}^{s}} 
\notag \\
&&{+g^{jb}R_{isr}^{\text{ \ \ \ }%
h}A_{lh}^{r}t_{b}^{s}-g_{ia}R_{.}^{s}{}_{.}^{j}{}_{r}^{\text{ }%
h}A_{lh}^{r}t_{s}^{a}-2f_{r}g_{ia}R_{.}^{s}{}_{.}^{j}{}_{l}^{\text{ }%
h}t_{s}^{a}}  \notag \\
&&{+2f_{r}g^{jb}R_{isl}^{\text{ \ \ \ }r}t_{b}^{s}),}  \notag \\
\widetilde{{R}}{_{lj}} &{=}&  \notag \\
&&{R_{l}{}_{j}-\frac{1}{4f}(g_{ka}R_{.}^{s}{}_{.}^{h}{}_{l}^{\text{ }%
r}R_{rjh}^{\text{ \ \ \ }p})t_{s}^{a}t_{p}^{k}-\frac{1}{2f}%
(g_{ka}R_{.}^{s}{}_{.}^{h}{}_{j}^{\text{ }r}R_{rlh}^{\text{ \ \ \ }%
p})t_{s}^{a}t_{p}^{k}}  \notag \\
&&{-\frac{1}{4f}(R_{lhr}^{\text{ \ \ \ }s}g_{va}R_{.}^{p}{}_{.}^{r}{}_{j}^{%
\text{ }h})t_{s}^{v}t_{p}^{a}-\frac{1}{4f}(g^{hb}R_{ksl}^{\text{ \ \ \ }%
r}R_{rjp}^{\text{ \ \ \ }k})t_{b}^{s}t_{h}^{p}}  \notag \\
&&{-\frac{1}{2f}(g^{hb}R_{ksj}^{\text{ \ \ \ }r}R_{rlp}^{\text{ \ \ \ }%
k})t_{b}^{s}t_{h}^{p}-\frac{1}{4f}(R_{lhs}^{\text{ \ \ \ }v}g^{rb}R_{vpj}^{%
\text{ \ \ \ }h})t_{r}^{s}t_{b}^{p}}  \notag \\
&&{+\frac{1}{2f}(g_{ka}R_{.}^{s}{}_{.}^{h}{}_{j}^{\text{ }r}R_{rlp}^{\text{
\ \ \ }k})t_{s}^{a}t_{h}^{p}+\frac{1}{2f}(g^{hb}R_{kpj}^{\text{ \ \ \ }%
r}R_{rlh}^{\text{ \ \ \ }s})t_{b}^{p}t_{s}^{k}}  \notag \\
&&{+\nabla _{r}(\frac{1}{2f}A_{lj}^{r})-\nabla _{l}(\frac{1}{2f}A_{rj}^{r})+%
\frac{1}{4f^{2}}A_{rs}^{r}A_{lj}^{s}-\frac{1}{4f^{2}}A_{ls}^{r}A_{rj}^{s}.} 
\notag
\end{eqnarray}%
with respect to the adapted frame. From (\ref{B3.5}) and (\ref{B3.16}), the
scalar curvature of the rescaled Sasaki type metric ${}^{S}g_{f}$ is given by%
\begin{eqnarray*}
{}\widetilde{r} &=&\frac{1}{f}r-{\frac{1}{4f^{2}}}%
g^{ab}g^{hk}g^{lj}g^{ti}R_{slhv}R_{pjkr}t_{a}^{s}t_{b}^{p} \\
&&-{\frac{1}{4f^{2}}}g_{cd}g^{lj}g^{hk}g^{rv}R_{rlh}^{\text{ \ \ \ }%
s}\,R_{vjk}^{\text{ \ \ \ }p}t_{s}^{c}t_{p}^{d}+{\frac{1}{2f^{2}}}R_{cpr}^{%
\text{ \ \ \ }h}R_{h}\,_{.\;.\;}^{r\,b\;s}\,t_{s}^{c}t_{b}^{p} \\
&&+\frac{1}{f}g^{lj}({\nabla _{r}(\frac{1}{2f}A_{lj}^{r})-\nabla _{l}(\frac{1%
}{2f}A_{rj}^{r})+\frac{1}{4f^{2}}A_{rs}^{r}A_{lj}^{s}-\frac{1}{4f^{2}}%
A_{ls}^{r}A_{rj}^{s})}
\end{eqnarray*}%
Thus we have the result as follows.

\begin{theorem}
\label{Theorem5}Let $(M,g)$ be a Riemannian manifold and $T_{1}^{1}(M)$ be
its tensor bundle with the rescaled Sasaki type metric ${}^{S}g_{f}$. Let $r$
be the scalar curvature of $g$ and ${}\widetilde{{r}}$ be the scalar
curvature of ${}^{S}g_{f}$. Then the following equation holds:%
\begin{eqnarray*}
{}\widetilde{r} &=&\frac{1}{f}r-{\frac{1}{4f^{2}}}%
g^{ab}g^{hk}g^{lj}g^{ti}R_{slhv}R_{pjkr}t_{a}^{s}t_{b}^{p} \\
&&-{\frac{1}{4f^{2}}}g_{cd}g^{lj}g^{hk}g^{rv}R_{rlh}^{\text{ \ \ \ }%
s}\,R_{vjk}^{\text{ \ \ \ }p}t_{s}^{c}t_{p}^{d}+{\frac{1}{2f^{2}}}R_{cpr}^{%
\text{ \ \ \ }h}R_{h}\,_{.\;.\;}^{r\,b\;s}\,t_{s}^{c}t_{b}^{p}{+}\text{ }^{f}%
{L.}
\end{eqnarray*}

where%
\begin{equation*}
^{f}L=\frac{1}{f}g^{lj}({\nabla _{r}(\frac{1}{2f}A_{lj}^{r})-\nabla _{l}(%
\frac{1}{2f}A_{rj}^{r})+\frac{1}{4f^{2}}A_{rs}^{r}A_{lj}^{s}-\frac{1}{4f^{2}}%
A_{ls}^{r}A_{rj}^{s}).}
\end{equation*}
\end{theorem}

From the Theorem \ref{Theorem5}, we have the following conclusion.

\begin{corollary}
Let $(M,g)$ be a Riemannian manifold and $T_{1}^{1}(M)$ be its tensor bundle
with the rescaled Sasaki type metric $^{S}g_{f}$. Then

(i) ${}\widetilde{{r}}{=0}$ if and only if $R=0$ and $^{f}{L=0}$

(ii) If $^{f}{L\neq 0}$, then $\widetilde{{r}}{\neq 0}$
\end{corollary}

Let now $\left( M,g\right) $, $n>2$ be a Riemannian manifold of constant
curvature $\kappa $, i.e.

\begin{equation*}
R_{kmj}^{\text{ \ \ \ }s}=\kappa \left( \delta _{k}^{s}g_{mj}-\delta
_{m}^{s}g_{kj}\right)
\end{equation*}%
and

\begin{equation*}
r=n\left( n-1\right) \kappa .
\end{equation*}%
Then, from Theorem \ref{Theorem5} we have

\noindent \textbf{\ }

\begin{equation*}
{}\widetilde{r}=\frac{1}{f}r-{\frac{1}{4f^{2}}}%
g^{ab}g^{hk}g^{vr}g_{lj}R_{hvs}^{\,\text{\ \ \ \ }l}R_{krp}^{\,\text{\ \ \ \ 
}j}t_{a}^{s}t_{b}^{p}
\end{equation*}

\begin{equation*}
-{\frac{1}{4f^{2}}}g_{cd}g^{lj}g^{hk}g^{rv}R_{rlh}^{\text{ \ \ \ }%
s}\,R_{vjk}^{\text{ \ \ \ }p}t_{s}^{c}t_{p}^{d}+{\frac{1}{2f^{2}}}%
g^{re}g^{bz}R_{cpr}^{\text{ \ \ \ }h}R_{hez}^{\text{\ \ \ \ }%
s}\,t_{s}^{c}t_{b}^{p}+^{f}{L}
\end{equation*}

\begin{equation*}
=\frac{1}{f}r-{\frac{1}{4f^{2}}}g^{ab}g^{hk}g^{vr}g_{lj}\left( \kappa
(\delta _{h}^{j}g_{vs}-\delta _{v}^{j}g_{hs})\kappa (\delta
_{k}^{l}g_{rp}-\delta _{r}^{l}g_{kp}\right) t_{a}^{s}t_{b}^{p}
\end{equation*}

\begin{equation*}
-{\frac{1}{4f^{2}}}g_{cd}g^{lj}g^{hk}g^{rv}\left( \kappa (\delta
_{r}^{s}g_{lh}-\delta _{l}^{s}g_{rh})\kappa (\delta _{v}^{p}g_{jk}-\delta
_{j}^{p}g_{vk}\right) \,t_{s}^{c}t_{p}^{d}
\end{equation*}

\begin{equation*}
+{\frac{1}{2f^{2}}}g^{re}g^{bz}\left( \kappa (\delta _{c}^{h}g_{pr}-\delta
_{p}^{h}g_{cr})\kappa (\delta _{h}^{s}g_{ez}-\delta _{e}^{s}g_{hz})\right)
\,t_{s}^{c}t_{b}^{p}+^{f}{L}
\end{equation*}

\begin{equation*}
=\frac{1}{f}n(n-1)\kappa -{\frac{1}{4f^{2}}}\kappa
^{2}ng^{ab}g_{sp}t_{a}^{s}t_{b}^{p}+{\frac{1}{4f^{2}}}\kappa
^{2}g^{ab}g_{sp}t_{a}^{s}t_{b}^{p}+{\frac{1}{4f^{2}}}\kappa
^{2}g^{ab}g_{sp}t_{a}^{s}t_{b}^{p}-{\frac{1}{4f^{2}}}\kappa
^{2}ng^{ab}g_{sp}t_{a}^{s}t_{b}^{p}
\end{equation*}

\begin{equation*}
-{\frac{1}{4f^{2}}}\kappa ^{2}ng_{cd}g^{rv}t_{r}^{c}t_{v}^{d}+{\frac{1}{%
4f^{2}}}\kappa ^{2}g_{cd}g^{rj}t_{r}^{c}t_{j}^{d}+{\frac{1}{4f^{2}}}\kappa
^{2}g_{cd}g^{lv}t_{l}^{c}t_{v}^{d}-{\frac{1}{4f^{2}}}\kappa
^{2}ng_{cd}g^{lj}t_{l}^{c}t_{j}^{d}
\end{equation*}

\begin{equation*}
+{\frac{1}{2f^{2}}}\kappa ^{2}\delta _{c}^{s}\delta
_{p}^{b}t_{s}^{c}t_{b}^{p}-{\frac{1}{2f^{2}}}\kappa ^{2}\delta
_{c}^{b}\delta _{p}^{s}t_{s}^{c}t_{b}^{p}-{\frac{1}{2f^{2}}}\kappa
^{2}\delta _{c}^{b}\delta _{p}^{s}t_{s}^{c}t_{b}^{p}+{\frac{1}{2f^{2}}}%
\kappa ^{2}\delta _{c}^{s}\delta _{p}^{b}t_{s}^{c}t_{b}^{p}+^{f}{L}
\end{equation*}

\begin{equation*}
=\frac{1}{f}n(n-1)\kappa -{\frac{1}{2f^{2}}}\kappa ^{2}\left\Vert
t\right\Vert ^{2}(n-1)-{\frac{1}{2f^{2}}}\kappa ^{2}\left\Vert t\right\Vert
^{2}(n-1)+\frac{1}{f^{2}}\kappa ^{2}t_{c}^{c}t_{p}^{p}-\frac{1}{f^{2}}\kappa
^{2}t_{p}^{c}t_{c}^{p}+^{f}{L}
\end{equation*}

\begin{equation*}
=\frac{1}{f}(n-1)\kappa \{n-\frac{1}{f}\left\Vert t\right\Vert ^{2}\kappa )+%
\frac{1}{f}\kappa ^{2}(trace\,t)^{2}-(trace\,t^{2})\}+^{f}{L}
\end{equation*}%
Thus we have

\begin{theorem}
\noindent \textbf{\ }Let $\left( M,g\right) $, $n>2$ be a Riemannian
manifold of constant curvature $\kappa $. Then the scalar curvature ${}%
\widetilde{r}$ of $\left( T_{1}^{1}\left( M\right) ,{}^{S}g_{f}\right) $ is%
\begin{equation*}
\widetilde{r}=\frac{1}{f}(n-1)\kappa (n-\frac{1}{f}\left\Vert t\right\Vert
^{2}\kappa )+\frac{1}{f}\kappa ^{2}\left(
(trace\,t)^{2}-(trace\,t^{2})\right) +^{f}{L},
\end{equation*}%
where $\left\Vert t\right\Vert ^{2}=g_{kl}g^{ij}t_{k}^{i}t_{l}^{j}$.
\end{theorem}

\section{Other metric connections of the rescaled Sasaki type metric}

Let $\nabla $ be a linear connection on an $n-$dimensional differentiable
manifold $M$. The connection $\nabla $ is symmetric if its torsion tensor
vanishes, otherwise it is non-symmetric. If there is a Riemannian metric $g$
on $M$ such that $\nabla g=0$, then the connection $\nabla $ is a metric
connection, otherwise it is non-metric. It is well known that a linear
connection is symmetric and metric if and only if it is the Levi-Civita
connection. In section 3, we have considered the Levi-Civita connection $%
\widetilde{\nabla }$ of the rescaled Sasaki type metric ${}^{S}g_{f}$ on the
tensor bundle $T_{1}^{1}(M)$ over $(M,g).$ The connection is the unique
connection which satisfies $\widetilde{\nabla }_{\alpha }(^{S}g_{f})_{\beta
\gamma }=0$ and has a zero torsion. H. A. Hayden \cite{Hayden} introduced a
metric connection with a non-zero torsion on a Riemanian manifold. Now we
are interested in a metric connection $^{(M)}\widetilde{\nabla }$ of the
rescaled Sasaki type metric ${}^{S}g_{f}$ whose torsion tensor $^{^{(M)}%
\widetilde{\nabla }}T_{\gamma \beta }^{\varepsilon }$ is skew-symmetric in
the indices $\gamma $ and $\beta .$ We denote components of the connection $%
^{(M)}\widetilde{\nabla }$ by $^{(M)}\widetilde{\Gamma }.$ The metric
connection $^{(M)}\widetilde{\nabla }$ satisfies 
\begin{equation}
^{(M)}\widetilde{\nabla }_{\alpha }(^{S}g_{f})_{\beta \gamma }=0\text{ and }%
^{(M)}\widetilde{\Gamma }_{\alpha \beta }^{\gamma }-\text{ }^{(M)}\widetilde{%
\Gamma }_{\beta \alpha }^{\gamma }=\text{ }^{^{(M)}\widetilde{\nabla }%
}T_{\alpha \beta }^{\gamma }  \label{B5.1}
\end{equation}%
On the equation (\ref{B5.1}) is solved with respect to $^{(M)}\widetilde{%
\Gamma }_{\alpha \beta }^{\gamma },$ one finds the following solution \cite%
{Hayden}%
\begin{equation}
^{(M)}\widetilde{\Gamma }_{\alpha \beta }^{\gamma }=\widetilde{\Gamma }%
_{\alpha \beta }^{\gamma }+\widetilde{U}_{\alpha \beta }^{\gamma }
\label{B5.2}
\end{equation}%
where $\widetilde{\Gamma }_{\alpha \beta }^{\gamma }$ is components of the
Levi-Civita connection of the rescaled Sasaki type metric ${}^{S}g_{f}$, 
\begin{equation}
\widetilde{U}_{\alpha \beta \gamma }=\frac{1}{2}(^{^{(M)}\widetilde{\nabla }%
}T_{\alpha \beta \gamma }+\text{ }^{^{(M)}\widetilde{\nabla }}T_{\gamma
\alpha \beta }+\text{ }^{^{(M)}\widetilde{\nabla }}T_{\gamma \beta \alpha })
\label{B5.3}
\end{equation}%
and%
\begin{equation*}
\widetilde{U}_{\alpha \beta \gamma }=\widetilde{U}_{\alpha \beta }^{\epsilon
}(^{S}g_{f})_{\epsilon \gamma }\text{, }^{^{(M)}\widetilde{\nabla }%
}T_{\alpha \beta \gamma }=T_{\alpha \beta }^{\epsilon }(^{S}g_{f})_{\epsilon
\gamma }.
\end{equation*}

If we put%
\begin{equation}
^{^{(M)}\widetilde{\nabla }}T_{lj}^{\overline{r}}=t_{r}^{m}R_{ljm}^{\text{ \
\ }v}-t_{m}^{v}R_{ljr}^{\text{ \ \ }m}  \label{B5.4}
\end{equation}%
all other $^{^{(M)}\widetilde{\nabla }}T_{\alpha \beta }^{\gamma }$ not
related to $^{^{(M)}\widetilde{\nabla }}T_{lj}^{\overline{r}}$ being assumed
to be zero. \ We choose this $^{^{(M)}\widetilde{\nabla }}T_{\alpha \beta
}^{\gamma }$ in $T_{1}^{1}(M)$ which is skew-symmetric in the indices $%
\gamma $ and $\beta $ as torsion tensor and determine a metric connection in 
$T_{1}^{1}(M)$ with respect to the rescaled Sasaki type metric ${}^{S}g_{f}.$
By using (\ref{B3.5}), (\ref{B5.3}) and (\ref{B5.4}), we get non-zero
components of $\widetilde{U}_{\alpha \beta }^{\gamma }$ as follows: 
\begin{eqnarray*}
\widetilde{{U}}{_{lj}^{\overline{r}}}{=\frac{1}{2}(t_{r}^{s}R_{ljs}^{\text{
\ \ }v}-t_{s}^{v}R_{ljr}^{\text{ \ \ }s}),} && \\
\widetilde{{U}}{_{l\overline{j}}^{r}}{=\frac{1}{2f}(g^{jb}R_{isl}^{\text{ \
\ }r}t_{b}^{s}-g_{ia}R_{.}^{s}{}_{.}^{j}{}_{l}^{\text{ }r}t_{s}^{a}),} && \\
\widetilde{{U}}{_{l\overline{j}}^{r}}{=\frac{1}{2f}(g^{lb}R_{tsj}^{\text{ \
\ }r}t_{b}^{s}-g_{ta}R_{.}^{s}{}_{.}^{l}{}_{j}^{\text{ }r}t_{s}^{a}),} &&
\end{eqnarray*}%
with respect to the adapted frame. From (\ref{B5.2}) and Theorem \ref%
{Theorem1}, we have

\begin{proposition}
Let $(M,g)$ be a Riemannian manifold and $T_{1}^{1}(M)$ be its tensor bundle
with the rescaled Sasaki type metric ${}^{S}g_{f}.$ The metric connection $%
^{(M)}\widetilde{\nabla }$ with respect to ${}^{S}g_{f}$ is as follows:%
\begin{equation*}
\left\{ 
\begin{array}{l}
i)\text{ }^{(M)}{\widetilde{{}\nabla }_{E_{l}}E_{j}=\{\Gamma _{lj}^{r}+\frac{%
1}{2f}{}^{f}A_{lj}^{r}\}E_{r},} \\ 
ii)\text{ }^{(M)}{\widetilde{{}\nabla }_{E_{l}}E_{\overline{j}}=\{\Gamma
_{li}^{v}\delta _{r}^{j}-\Gamma _{lr}^{j}\delta _{i}^{v}\}E_{\overline{r}},}
\\ 
iii)\text{ }^{(M)}{\widetilde{{}\nabla }_{E_{\overline{l}}}E_{j}=0,} \\ 
iv)\text{ }^{(M)}{\widetilde{{}\nabla }_{E_{\overline{l}}}E_{\overline{j}}=0}%
\end{array}%
\right.
\end{equation*}%
with respect to the adapted frame, where ${}^{f}A_{ji}^{h}$ is a symmetric
tensor field of type $(1,2)$ defined by ${}^{f}A_{ji}^{h}=(f_{j}\delta
_{i}^{h}+f_{i}\delta _{j}^{h}-f_{.}^{h}g_{ji}).$
\end{proposition}

\begin{remark}
If $f=C(const.)$, the metric connection $^{(M)}{\widetilde{{}\nabla }}$ in $%
T_{1}^{1}(M)$ of the rescaled Sasaki type metric ${}^{S}g_{f}$ coincides
with the metric connection $^{H}\nabla $ of the Sasaki type metric $^{S}g$,
where $^{H}\nabla $ is the horizontal lift of the Levi-Civita connection $%
\nabla $ of $g$ (for the metric connection $^{H}\nabla $, see \cite{Salimov3}%
).
\end{remark}

For the curvature tensor $^{(M)}{\widetilde{{}R}}$ of the metric connection $%
^{(M)}{\widetilde{{}\nabla }}$, we state the following result.

\begin{proposition}
Let $(M,g)$ be a Riemannian manifold and $T_{1}^{1}(M)$ be its tensor bundle
with the rescaled Sasaki type metric ${}^{S}g_{f}.$ The curvature tensor $%
^{(M)}{\widetilde{{}R}}$ of the metric connection $^{(M)}{\widetilde{%
{}\nabla }}$ satisfies the followings:%
\begin{equation*}
\left\{ 
\begin{array}{l}
^{(M)}{\widetilde{{}R}(E_{m},E_{l})E_{j}=\{R_{mlj}^{\text{ \ \ \ }r}+\nabla
_{m}(\frac{1}{2f}A_{lj}^{r})-\nabla _{l}(\frac{1}{2f}A_{mj}^{r})} \\ 
{+\frac{1}{4f^{2}}A_{ms}^{r}A_{lj}^{s}-\frac{1}{4f^{2}}A_{ls}^{r}A_{mj}^{s}%
\}E_{r},} \\ 
^{(M)}{\widetilde{{}R}(E_{m},E_{l})E_{\overline{j}}=\{R_{mli}^{\text{ \ \ \ }%
v}\delta _{r}^{j}-R_{mlr}^{\text{ \ \ \ }j}\delta _{i}^{v}\}E_{\overline{r}},%
} \\ 
otherwise=0%
\end{array}%
\right.
\end{equation*}%
with respect to the adapted frame.
\end{proposition}

The non-zero component of the contracted curvature tensor field (Ricci
tensor field) $^{(M)}{\widetilde{{}R}}_{\gamma \beta }=$ $^{(M)}{\widetilde{%
{}R}}_{\alpha \gamma \beta }^{\text{ \ \ \ }\alpha }$ of the metric
connection $^{(M)}{\widetilde{{}\nabla }}$ is as follows:%
\begin{equation*}
^{(M)}{\widetilde{{}R}_{lj}=R_{l}{}_{j}+\nabla _{r}(\frac{1}{2f}%
A_{lj}^{r})-\nabla _{l}(\frac{1}{2f}A_{rj}^{r})+\frac{1}{4f^{2}}%
A_{rs}^{r}A_{lj}^{s}-\frac{1}{4f^{2}}A_{ls}^{r}A_{rj}^{s}.}
\end{equation*}%
For the scalar curvature $^{(M)}\widetilde{r}$ of the metric connection $%
^{(M)}{\widetilde{{}\nabla }}$ with respect to $^{S}g_{f}$ , we obtain%
\begin{eqnarray*}
^{(M)}\widetilde{r} &=&(^{S}g_{f})^{\gamma \beta }\text{ }^{(M)}{\widetilde{%
{}R}}_{\gamma \beta } \\
&=&\frac{1}{f}r+^{f}{L}
\end{eqnarray*}%
where $^{f}{L=}\frac{1}{f}g^{lj}\{\nabla _{r}(\frac{1}{2f}A_{lj}^{r})-\nabla
_{l}(\frac{1}{2f}A_{rj}^{r})+\frac{1}{4f^{2}}A_{rs}^{r}A_{lj}^{s}-\frac{1}{%
4f^{2}}A_{ls}^{r}A_{rj}^{s}$ and $r$ is the scalar curvature of $\nabla
_{g}. $Thus we have the following theorem.

\begin{theorem}
Let $(M,g)$ be a Riemannian manifold, and the tensor bundle $T_{1}^{1}(M)$
be equipped with the rescaled Sasaki type metric $^{S}g_{f}.$ Then the
tensor bundle $T_{1}^{1}(M)$ with the metric connection $^{(M)}{\widetilde{%
{}\nabla }}$ $\ $has vanishing scalar curvature $^{(M)}\widetilde{r}$ with
respect to $^{S}g_{f}$ if and only if the scalar curvature $r$ of $\nabla
_{g}$ in $M$ is zero and $^{f}{L}=0.$\bigskip
\end{theorem}

For a Riemannian metric $g$ on $M$, there happen to be many ways to define
metric connections associated with $g.$ Now we shall give another class of
metric connections on $T_{1}^{1}(M).$ Let $F$ be an almost product structure
and $\nabla $ be a linear connection on an $n$-dimensional Riemannian
manifold $M.$ The product conjugate connection $^{(F)}\nabla $ of $\nabla $
is defined by%
\begin{equation*}
^{(F)}\nabla _{X}Y=F(\nabla _{X}FY)
\end{equation*}%
for all $X,Y\in \Im _{0}^{1}(M).$ If $(M,F,g)$ is an almost product
Riemannian manifold, then ($^{(F)}\nabla _{X}g)(FY,FZ)=(\nabla _{X}g)(Y,Z)$,
i.e. $\nabla $ is a metric connection with respect to $g$ if and only if $%
^{(F)}\nabla $ is so. From this, we can say that if $\nabla $ is the
Levi-Civita connection of $g$, then $^{(F)}\nabla $ is a metric connection
with respect to $g$ \cite{Blaga}$.$

By the almost product structure ${J}$ defined by (\ref{B3.8}) and the
Levi-Civita connection $\widetilde{\nabla }$ given by Theorem \ref{Theorem1}%
, we write the product conjugate connection $^{(J)}\widetilde{\nabla }$ of $%
\widetilde{\nabla }$ as follows:%
\begin{equation*}
^{(J)}\widetilde{\nabla }_{\widetilde{X}}\widetilde{Y}=J(\widetilde{\nabla }%
_{\widetilde{X}}J\widetilde{Y})
\end{equation*}%
for all $\widetilde{X},\widetilde{Y}\in \Im _{0}^{1}(T_{1}^{1}(M)).$ Also
note that $^{(J)}\widetilde{\nabla }$ is a metric connection of the rescaled
Sasaki type metric $^{S}g_{f}$. The standart calculations give the following
theorem.

\begin{theorem}
Let $(M,g)$ be a Riemannian manifold and let $T_{1}^{1}(M)$ be its tensor
bundle equipped with the rescaled Sasaki type metric $^{S}g_{f}$ and the
almost product structure $J.$ Then the product conjugate connection (or
metric connection) $^{(J)}\widetilde{\nabla }$ is as follows: 
\begin{equation*}
\left\{ 
\begin{array}{l}
i)\text{ }^{(J)}{\widetilde{\nabla }_{E_{l}}E_{j}=\{\Gamma _{lj}^{r}+\frac{1%
}{2f}{}^{f}A_{lj}^{r}\}E_{r}-\{\frac{1}{2}R_{ljr}^{\text{ \ \ }s}t_{s}^{v}-%
\frac{1}{2}R_{ljs}^{\text{ \ \ }v}t_{r}^{s}\}E_{\overline{r}},} \\ 
ii)\text{ }^{(J)}{\widetilde{\nabla }_{E_{l}}E_{\overline{j}}=-\{\frac{1}{2f}%
g_{ia}R_{.}^{s}{}_{.}^{j}{}_{l}^{\text{ }r}t_{s}^{a}-\frac{1}{2f}%
g^{jb}R_{isl}^{\text{ \ \ }r}t_{b}^{s}\}E_{r}+\{\Gamma _{li}^{v}\delta
_{r}^{j}-\Gamma _{lr}^{j}\delta _{i}^{v}\}E_{\overline{r}},} \\ 
iii)\text{ }^{(J)}{\widetilde{\nabla }_{E_{\overline{l}}}E_{j}=\{\frac{1}{2f}%
g_{ta}R_{.}^{s}{}_{.}^{l}{}_{j}^{\text{ }r}t_{s}^{a}-\frac{1}{2f}%
g^{lb}R_{tsj}^{\text{ \ \ }r}t_{b}^{s}\}E_{r},} \\ 
iv)\text{ }^{(J)}{\widetilde{\nabla }_{E_{\overline{l}}}E_{\overline{j}}=0}%
\end{array}%
\right.
\end{equation*}%
with respect to the adapted frame.
\end{theorem}

The relationship between curvature tensors $R_{\nabla }$ and $%
R_{^{(F)}\nabla }$ of the connections $\nabla $ and $^{(F)}\nabla $ is
follows: $R_{^{(F)}\nabla }(X,Y,Z)=F(R_{\nabla }(X,Y,FZ)$ for all $X,Y,Z\in
\Im _{0}^{1}(M)$ \cite{Blaga}$.$ By means of the almost product structure ${J%
}$ defined by (\ref{B3.8}) and Theorem \ref{Theorem1}, from $\widetilde{R}%
_{^{(J)}\widetilde{\nabla }}(\widetilde{X},\widetilde{Y},\widetilde{Z})=J(%
\widetilde{R}_{\widetilde{\nabla }}(\widetilde{X},\widetilde{Y},J\widetilde{Z%
})$, components of the curvature tensor $\widetilde{R}_{^{(J)}\widetilde{%
\nabla }}$ of the product conjugate connection (or metric connection) $^{(J)}%
\widetilde{\nabla }$ can easily be computed. Lastly, by using the almost
product structure $^{D}I$, another metric connection of the rescaled Sasaki
type metric $^{S}g_{f}$ can be constructed.

\section{The Rescaled Sasaki type metric on $(p,q)-$tensor bundles and its
Geodesics}

The set $T_{q}^{p}(M)=\underset{P\in M}{\cup }T_{q}^{p}(P)$ is $(p,q)$
tensor bundles over $M$, where $T_{q}^{p}(P)$ is tensor spaces for all $P\in
M$. For $\tilde{P}\in T_{q}^{p}(M)$, the surjective correspondence $\tilde{P}%
\rightarrow P$ determines the natural projection $\pi
:T_{q}^{p}(M)\rightarrow M.$ A system of local coordinates $\left(
U,x^{j}\right) ,\mathrm{\;}j=1,...,n$ in $M$ induces on ${}T_{q}^{p}(M)$ a
system of local coordinates $\left( \pi ^{-1}\left( U\right) ,\mathrm{\;}%
x^{j},\mathrm{\;}x^{\overline{j}}=t_{j_{1}...j_{q}}^{i_{1}...i_{p}}\right) ,%
\mathrm{\;}\overline{j}=n+1,...,n+n^{p+q}$, where $x^{\overline{j}%
}=t_{j_{1}...j_{q}}^{i_{1}...i_{p}}$ is the components of tensors $t$ in
each tensor space ${}T_{q}^{p}(M)_{x},\mathrm{\;}x\in U$ with respect to the
natural base.

The vertical lift $^{V}A$ of $A\in \Im _{q}^{p}(M)$ and the horizontal lift $%
^{H}X$ of $X\in \Im _{0}^{1}(M)$ to $T_{q}^{p}(M)$ are given by\noindent 
\begin{equation*}
{}^{V}A=\left( 
\begin{array}{l}
{{}^{V}A^{j}} \\ 
{{}^{V}A^{\overline{j}}}%
\end{array}%
\right) =\left( 
\begin{array}{c}
{0} \\ 
{A_{j_{1}...j_{q}}^{i_{1}...i_{p}}}%
\end{array}%
\right) ,
\end{equation*}%
and

\begin{equation*}
^{H}X=\left( 
\begin{array}{c}
{{}^{H}X^{j}} \\ 
{{}^{H}X^{\overline{j}}}%
\end{array}%
\right) =\left( 
\begin{array}{c}
X{^{j}} \\ 
X{^{s}}(\sum\limits_{\mu =1}^{q}\Gamma _{sj_{\mu
}}^{m}t_{j_{1}...m...j_{q}}^{i_{1}...i_{p}}-\sum\limits_{\lambda
=1}^{p}\Gamma _{sm}^{i_{\lambda }}t_{j_{1}...j_{q}}^{i_{1}...m...i_{p}})%
\end{array}%
\right) ,
\end{equation*}%
where $\Gamma _{ij}^{h}$ are the coefficients of the Levi-Civita connection $%
\nabla $ of $g$ \cite{Cengiz}. For $\varphi =\varphi _{j}^{i}\frac{\partial 
}{\partial x^{i}}\otimes dx^{j}\in \Im _{1}^{1}(M),$\ the local expressions
of the global vector fields $\gamma \varphi $ and $\widetilde{\gamma }%
\varphi $ are as follows:%
\begin{equation*}
\gamma \varphi =\left( 
\begin{array}{c}
0 \\ 
\sum\limits_{\lambda =1}^{p}t_{j_{1}...j_{q}}^{i_{1}...m...i_{p}}\varphi
_{m}^{i_{\lambda }}%
\end{array}%
\right) ~~~\text{and}~~~\widetilde{\gamma }\varphi =\left( 
\begin{array}{c}
0 \\ 
\sum\limits_{\mu =1}^{q}t_{j_{1}...m...j_{q}}^{i_{1}...i_{p}}\varphi
_{j_{\mu }}^{m}%
\end{array}%
\right) .
\end{equation*}

Now, we define the adapted frame $\left\{ E_{\alpha }\right\} =\left\{
E_{j},E_{\overline{j}}\right\} $ of $T_{q}^{p}(M)$ as follows:%
\begin{equation*}
E_{j}=\text{ }^{H}X_{(j)}=\delta _{j}^{h}\partial
_{h}+(-\tsum\limits_{\lambda =1}^{p}\Gamma _{js}^{k_{\lambda
}}t_{h_{1}...h_{q}}^{k_{1}...s...k_{p}}+\tsum\limits_{\mu =1}^{q}\Gamma
_{jh_{\mu }}^{s}t_{h_{1}...s...h_{q}}^{k_{1}...k_{p}})\partial _{\overline{h}%
}~,
\end{equation*}%
\begin{equation*}
E_{\overline{j}}=^{V}A^{(\overline{j})}=\delta _{i_{1}}^{k_{1}}...\delta
_{i_{p}}^{k_{p}}\delta _{h_{1}}^{j_{1}}...\delta _{h_{q}}^{j_{q}}\partial _{%
\overline{h}}
\end{equation*}%
with respect to the natural frame $\left\{ \frac{\partial }{\partial x^{h}},%
\frac{\partial }{\partial x^{\overline{h}}}\right\} $ in $T_{q}^{p}(M),$
where $X_{(j)}=\frac{\partial }{\partial x^{j}}=\delta _{j}^{h}\frac{%
\partial }{\partial x^{h}}\in \Im _{0}^{1}(M)$ and $A^{(\overline{j}%
)}=\partial _{i_{1}}\otimes ...\otimes \partial _{i_{p}}\otimes
dx^{j_{1}}\otimes ...\otimes dx^{j_{q}}=\delta _{i_{1}}^{k_{1}}...\delta
_{i_{p}}^{k_{p}}\delta _{h_{1}}^{j_{1}}...\delta _{h_{q}}^{j_{q}}\partial
_{k_{1}}\otimes ...\otimes \partial _{k_{p}}\otimes dx^{h_{1}}\otimes
...\otimes dx^{h_{q}}\in \Im _{q}^{p}(M).$

With respect to the adapted frame $\left\{ E_{\alpha }\right\} ,$ the
vertical lift $^{V}A$ and the horizontal lift $^{H}X$ are as follows \cite%
{Salimov5}:%
\begin{equation*}
{}^{V}A=\left( 
\begin{array}{c}
{0} \\ 
{A_{j_{1}...j_{q}}^{i_{1}...i_{p}}}%
\end{array}%
\right) ,
\end{equation*}%
and

\begin{equation*}
^{H}X=\left( 
\begin{array}{c}
X{^{j}} \\ 
0%
\end{array}%
\right) .
\end{equation*}

The rescaled Sasaki type metric $^{S}g_{f}$ is defined on $T_{q}^{p}(M)$ by
the three equations%
\begin{equation*}
^{S}g_{f}(^{V}A,^{V}B)=^{V}(G(A,B)),
\end{equation*}%
\begin{equation*}
^{S}g_{f}(^{V}A,^{H}Y)=0,
\end{equation*}%
\begin{equation*}
^{S}g_{f}(^{H}X,^{H}Y)=^{V}(fg(X,Y)),
\end{equation*}%
for all $X,Y\in \Im _{0}^{1}(M)$ and $A,B\in \Im _{q}^{p}(M),$where 
\begin{equation*}
G(A,B)=g_{i_{1}t_{1}}...g_{i_{p}t_{p}}g^{j_{1}l_{1}}...g^{j_{q}l_{q}}A_{j_{1}...j_{q}}^{i_{1}...i_{p}}B_{l_{1}...l_{q}}^{t_{1}...t_{p}}.
\end{equation*}

The rescaled Sasaki type metric $^{S}g_{f}$ has components%
\begin{equation*}
(^{S}g_{f})_{\beta \gamma }=\left( 
\begin{array}{cc}
fg_{jl} & 0 \\ 
0 & g_{i_{1}t_{1}}...g_{i_{p}t_{p}}g^{j_{1}l_{1}}...g^{j_{q}l_{q}}%
\end{array}%
\right) ~,~x^{\overline{l}}=t_{l_{1}...l_{q}}^{t_{1}...t_{p}},
\end{equation*}%
with respect to the adapted frame, $g_{ij\text{ }}$and $g^{ij}$ being local
covariant and contravariant components $g$ of on $M$. By the Koszul formula
and standart calculations give the following.

\begin{proposition}
\label{Proposition3}The components of the Levi-Civita connection $\widehat{%
\nabla }$ of the tensor bundle $T_{q}^{p}(M)$ with the rescaled Sasaki type
metric\textit{\ ${}$}$^{S}g_{f}$ are given as follows:%
\begin{equation*}
\left\{ 
\begin{array}{l}
{}\widehat{{\Gamma }}{_{lj}^{\overline{r}}=}\frac{1}{2}\tsum\limits_{\mu
=1}^{q}{R_{ljh_{\mu }}^{\text{ \ \ \ }%
s}t_{h_{1}...s...h_{q}}^{k_{1}...k_{p}}-\frac{1}{2}}\tsum\limits_{\lambda
=1}^{p}{R_{ljs}^{\text{ \ \ \ }k_{\lambda
}}t_{h_{1}...h_{q}}^{k_{1}...s...k_{p}},} \\ 
\widehat{{\Gamma }}{_{l\overline{j}}^{\overline{r}}=}\tsum\limits_{\lambda
=1}^{p}{\Gamma _{ls}^{v_{\lambda }}\delta _{r_{1}}^{j_{1}}...\delta
_{r_{q}}^{j_{q}}\delta _{i_{1}}^{l_{1}}...\delta _{i_{\lambda
}}^{s}...\delta _{i_{p}}^{l_{p}}-}\tsum\limits_{\mu =1}^{q}{\Gamma _{lr_{\mu
}}^{s}\delta _{r_{1}}^{j_{1}}...\delta _{s}^{j_{\mu }}...\delta
_{r_{q}}^{j_{q}}\delta _{i_{1}}^{l_{1}}...\delta _{i_{p}}^{l_{p}},} \\ 
{}\widehat{{\Gamma }}{_{\overline{l}j}^{r}=\frac{1}{2f}%
g^{xr}g_{i_{1}t_{1}}...g_{i_{p}t_{p}}g^{j_{1}h_{1}}...g^{j_{q}h_{q}}(}%
\tsum\limits_{\mu =1}^{q}{R_{xjh_{\mu }}^{\text{ \ \ \ }%
s}t_{h_{1}...s...h_{q}}^{k_{1}...k_{p}}-}\tsum\limits_{\lambda =1}^{p}{%
R_{xjs}^{\text{ \ \ \ }k_{\lambda }}t_{h_{1}...h_{q}}^{k_{1}...s...k_{p}}),}
\\ 
\widehat{{\Gamma }}{_{l\overline{j}}^{r}=\frac{1}{2f}%
g^{xr}g_{t_{1}k_{1}}...g_{t_{p}k_{p}}g^{l_{1}h_{1}}...g^{l_{q}h_{q}}(}%
\tsum\limits_{\mu =1}^{q}{R_{xlh_{\mu }}^{\text{ \ \ \ }%
s}t_{h_{1}...s...h_{q}}^{k_{1}...k_{p}}-}\tsum\limits_{\lambda =1}^{p}{%
R_{xls}^{\text{ \ \ \ }k_{\lambda }}t_{h_{1}...h_{q}}^{k_{1}...s...k_{p}}),}
\\ 
{}\widehat{{\Gamma }}{_{lj}^{r}=\Gamma _{lj}^{r}+\frac{1}{2f}{}%
^{f}A_{lj}^{r},} \\ 
\widehat{{\Gamma }}{_{\overline{l}\overline{j}}^{\overline{r}}=0,{}}\widehat{%
{\Gamma }}{_{\overline{l}\overline{j}}^{r}=0,{}}\widehat{{\Gamma }}{_{%
\overline{l}j}^{\overline{r}}=0,}%
\end{array}%
\right.
\end{equation*}%
with respect to the adapted frame, where ${}^{f}A_{ji}^{h}$ is defined by $%
{}^{f}A_{ji}^{h}=(f_{j}\delta _{i}^{h}+f_{i}\delta _{j}^{h}-f_{.}^{h}g_{ji})$
and $f_{i}=\partial _{i}f.$
\end{proposition}

An important geometric problem is to find the geodesics on the smooth
manifolds with respect to the Riemannian metrics (see \cite{Caddeo, Dusek,
Dusekk, Nagy, Salimov5, YanoIshihara:DiffGeo}. In \cite{YanoIshihara:DiffGeo}%
, K. Yano and S. Ishihara proved that the curves on the tangent bundles of
Riemannian manifolds are geodesics with respect to certain lifts of the
metric from the base manifold, if and only if the curves are obtained as
certain types of lifts of the geodesics from the base manifold. In this
section, we shall characterize the geodesics on the $(p,q)$-tensor bundle
with respect to the Levi-Civita connection of \textit{${}$}$^{S}g_{f}$ and a
metric connection of \textit{${}$}$^{S}g_{f}.$

Let $\widetilde{\gamma }=\widetilde{\gamma }(t)$ be a curve in $T_{q}^{p}(M)$
and suppose that $\widetilde{\gamma }$ is locally expressed by $%
x^{R}=x^{R}(t)$, i.e. $x^{r}=x^{r}(t)$, $x^{\overline{r}%
}=t_{r_{1}...r_{q}}^{v_{1}...v_{p}}(t)$ with respect to the natural frame $%
\left\{ \frac{\partial }{\partial x^{l}}\right\} =\left\{ \frac{\partial }{%
\partial x^{i}},\frac{\partial }{\partial x^{\overline{i}}}\right\} $, $t$
being a parameter an arc length of $\widetilde{\gamma }$. Then the curve $%
\gamma =\pi \circ \widetilde{\gamma }$ on $M$ is called the projection of
the curve $\widetilde{\gamma }$ and denoted by $\pi \widetilde{\gamma }$
which is expressed locally by $x^{r}=x^{r}(t)$.

A curve $\widetilde{\gamma }$ is, by definition, a geodesic in $T_{q}^{p}(M)$
with respect to the Levi-Civita connection $\widehat{\nabla }$ of $^{S}g_{f}$
if and only if it satisfies the differential equations%
\begin{equation}
\frac{d}{dt}(\frac{\omega ^{\varepsilon }}{dt})+\widehat{\Gamma }_{\gamma
\beta }^{\alpha }\frac{\omega ^{\gamma }}{dt}\frac{\omega ^{\beta }}{dt}=0
\label{B6.1}
\end{equation}%
with respect to the adapted frame, where 
\begin{equation*}
\frac{\omega ^{r}}{dt}=\frac{dx^{r}}{dt},~\ \ \frac{\omega ^{\overline{r}}}{%
dt}=\frac{\delta t_{r_{1}...r_{q}}^{v_{1}...v_{p}}}{dt},
\end{equation*}%
along a curve $\widetilde{\gamma }$.

By means of Proposition \ref{Proposition3}, (\ref{B6.1}) reduces to

\begin{eqnarray}
&&\frac{d}{dt}(\frac{\omega ^{r}}{dt})+({\Gamma _{lj}^{r}+\frac{1}{2f}{}%
^{f}A_{lj}^{r})}\frac{dx^{l}}{dt}\frac{dx^{j}}{dt}  \label{B6.2} \\
&&+{\frac{1}{2f}}%
g^{xr}g_{t_{1}k_{1}}...g_{t_{p}k_{p}}g^{l_{1}h_{1}}...g^{l_{q}h_{q}}(-\tsum%
\limits_{\lambda =1}^{p}R_{xjs}~^{k_{\lambda
}}t_{h_{1}...h_{q}}^{k_{1}...s...k_{p}}  \notag \\
&&+\tsum\limits_{\mu =1}^{q}R_{xjh_{\mu
}}~^{s}t_{h_{1}...s...h_{q}}^{k_{1}...k_{p}})\frac{\delta
_{l_{1}...l_{q}}^{t_{1}...t_{p}}}{dt}\frac{dx^{j}}{dt}  \notag \\
&&+\frac{1}{2f}%
g^{xr}g_{i_{1}k_{1}}...g_{i_{p}k_{p}}g^{j_{1}h_{1}}...g^{j_{q}h_{q}}(-\tsum%
\limits_{\lambda =1}^{p}R_{xls}~^{k_{\lambda
}}t_{h_{1}...h_{q}}^{k_{1}...s...k_{p}}  \notag \\
&&+\tsum\limits_{\mu =1}^{q}R_{xlh_{\mu
}}~^{s}t_{h_{1}...s...h_{q}}^{k_{1}...k_{p}})\frac{dx^{l}}{dt}\frac{\delta
t_{j_{1}...j_{q}}^{l_{1}...l_{p}}}{dt}  \notag \\
&=&0,  \notag
\end{eqnarray}

\begin{eqnarray}
&&\frac{d}{dt}(\frac{\delta t_{r_{1}...r_{q}}^{v_{1}...v_{p}}}{dt})+\frac{1}{%
2}(\tsum\limits_{\mu =1}^{q}R_{ljh_{\mu
}}~^{s}t_{h_{1}...s...h_{q}}^{k_{1}...k_{p}}-\tsum\limits_{\lambda
=1}^{p}R_{ljs}~^{k_{\lambda }}t_{h_{1}...h_{q}}^{k_{1}...s...k_{p}})\frac{%
dx^{l}}{dt}\frac{dx^{j}}{dt}  \label{B6.3} \\
&&+(\tsum\limits_{\lambda =1}^{p}\Gamma _{ls}^{v_{\lambda }}\delta
_{r_{1}}^{j_{1}}...\delta _{r_{q}}^{j_{q}}\delta _{i_{1}}^{l_{1}}...\delta
_{i_{\lambda }}^{s}...\delta _{i_{p}}^{l_{p}}-\tsum\limits_{\mu
=1}^{q}\Gamma _{lr_{\mu }}^{s}\delta _{r_{1}}^{j_{1}}...\delta _{s}^{j_{\mu
}}...\delta _{r_{q}}^{j_{q}}\delta _{i_{1}}^{l_{1}}...\delta
_{i_{p}}^{l_{p}})\frac{dx^{l}}{dt}\frac{\delta
t_{j_{1}...j_{q}}^{i_{1}...i_{p}}}{dt}  \notag \\
&=&0.  \notag
\end{eqnarray}

\bigskip We now transform (\ref{B6.2}) as follows:

\begin{eqnarray}
&&\frac{\delta ^{2}x^{r}}{dt^{2}}{+\frac{1}{2f}{}^{f}A_{lj}^{r}}\frac{dx^{l}%
}{dt}\frac{dx^{j}}{dt}+\frac{1}{f}%
g_{t_{1}k_{1}}...g_{t_{p}k_{p}}g^{l_{1}h_{1}}...g^{l_{q}h_{q}}(\tsum%
\limits_{\mu =1}^{q}R_{.jh_{\mu }}^{r\text{ \ \ \ }%
s}~t_{h_{1}...s...h_{q}}^{k_{1}...k_{p}}  \label{B6.4} \\
&&-\tsum\limits_{\lambda =1}^{p}R_{.js}^{r\text{ \ \ \ }k_{\lambda
}}~t_{h_{1}...h_{q}}^{k_{1}...s...k_{p}})\frac{\delta
_{l_{1}...l_{q}}^{t_{1}...t_{p}}}{dt}\frac{dx^{j}}{dt}  \notag \\
&=&0.  \notag
\end{eqnarray}

\bigskip

Using the identity $\left( -\sum\limits_{\lambda =1}^{p}R_{ljs}^{\text{ \ \ }%
k_{\lambda }}t_{h_{1}...h_{q}}^{k_{1}...s...k_{p}}+\sum\limits_{\mu
=1}^{q}R_{ljh_{\mu }}^{\text{ \ \ \ }s}t_{h_{1}...s...h_{q}}^{k_{1}...k_{p}}%
\right) \frac{dx^{l}}{dt}\frac{dx^{j}}{dt}=0$, transform of (\ref{B6.3}) as
follow:%
\begin{equation*}
\frac{\delta ^{2}t_{r_{1}...r_{q}}^{v_{1}...v_{p}}}{dt^{2}}=0.
\end{equation*}

Hence, we have the theorem below.

\begin{theorem}
Let $\widetilde{\gamma }$ be a geodesic on $T_{q}^{p}(M)$ of the Levi-Civita
connection $\widehat{\nabla }$ of $^{S}g_{f}$. Then the tensor field $%
t_{r_{1}...r_{q}}^{v_{1}...v_{p}}(t)$ defined along $\gamma $ satisfies the
differential equations (\ref{B6.4}) and has vanishing second covariant
derivative.
\end{theorem}

Next, let $\gamma $ be a curve on $M$ expressed locally by $x^{h}=x^{h}(t)$
and $S_{j_{1}...j_{q}}^{i_{1}...i_{p}}(t)$ be a $\left( p,q\right) $ tensor
field along $\gamma $. Then, on the tensor bundle $T_{q}^{p}(M)$ over the
Riemannian manifold $M$, we define a curve $^{H}\gamma $ by%
\begin{equation*}
\left\{ 
\begin{array}{c}
x^{h}=x^{h}(t), \\ 
x^{\overline{h}}=S_{h_{1}...h_{q}}^{k_{1}...k_{p}}(t).%
\end{array}%
\right.
\end{equation*}%
If the curve $^{H}\gamma $ satisfies at all the points the relation%
\begin{equation}
\frac{\delta S_{h_{1}...h_{q}}^{k_{1}...k_{p}}}{dt}=0,  \label{B6.5}
\end{equation}%
i.e. $S_{j_{1}...j_{q}}^{i_{1}...i_{p}}(t)$ is a parallel tensor field along 
$\gamma $, then the curve $^{H}\gamma $ is said to be a horizontal lift of $%
\gamma $. From (\ref{B6.4}) and (\ref{B6.5}), we obtain%
\begin{equation*}
\frac{\delta ^{2}x^{r}}{dt^{2}}{+\frac{1}{2f}{}^{f}A_{lj}^{r}}\frac{dx^{l}}{%
dt}\frac{dx^{j}}{dt}=0.
\end{equation*}%
If we take ${}$%
\begin{equation}
{^{f}A_{lj}^{r}=(\partial }_{l}{f\delta _{j}^{r}+\partial }_{j}{f\delta
_{l}^{r}-g}^{rm}{\partial }_{m}f{g_{lj})=0}  \label{B6.6}
\end{equation}%
Contracting $l$ and $r$ in (\ref{B6.6}) it follows that ${\partial }_{j}{f=0.%
}$ Since this is true for any $j,$ we can say $f=C(const.).$ Thus we have
the following theorem.

\begin{theorem}
The horizontal lift of a geodesic on $M$ is always geodesic on $T_{q}^{p}(M)$
with respect to the Levi-Civita connection $\widehat{\nabla }$ of $^{S}g_{f}$
if and only if $f=C(const.).$
\end{theorem}

By the same way in the section 5, by using the Levi-Civita connection $%
\widehat{\nabla }$ of $^{S}g_{f}$ in $T_{q}^{p}(M)$, we introduce a metric
connection $^{(M)}\widehat{\nabla }$ in $T_{q}^{p}(M)$ whose torsion tensor
has components%
\begin{equation*}
\left\{ 
\begin{array}{c}
^{^{(M)}\widehat{\nabla }}T_{lj}^{\overline{r}}={\frac{1}{2}}%
\tsum\limits_{\lambda =1}^{p}{R_{ljs}^{\text{ \ \ }k_{\lambda
}}t_{h_{1}...h_{q}}^{k_{1}...s...k_{p}}-\frac{1}{2}}\tsum\limits_{\mu =1}^{q}%
{R_{ljh_{\mu }}^{\text{ \ \ }s}t_{h_{1}...s...h_{q}}^{k_{1}...k_{p}}}, \\ 
otherwise=0%
\end{array}%
\right.
\end{equation*}%
with respect to $^{S}g_{f}.$

\begin{proposition}
Let $(M,g)$ be a Riemannian manifold and $T_{q}^{p}(M)$ be its $(p,q)-$%
tensor bundle with the rescaled Sasaki type metric ${}^{S}g_{f}.$ The
components of the metric connection $^{(M)}\widehat{\nabla }$ with respect
to ${}^{S}g_{f}$ is as follows:%
\begin{equation*}
\left\{ 
\begin{array}{l}
^{(M)}{}\widehat{{\Gamma }}{_{lj}^{\overline{r}}=0,}{}^{(M)}\widehat{{\Gamma 
}}{_{\overline{l}j}^{r}=0,}^{(M)}\widehat{{\Gamma }}{_{l\overline{j}}^{r}=0,}%
^{(M)}\widehat{{\Gamma }}{_{\overline{l}\overline{j}}^{\overline{r}}=0,{}}%
^{(M)}\widehat{{\Gamma }}{_{\overline{l}\overline{j}}^{r}=0,{}}^{(M)}%
\widehat{{\Gamma }}{_{\overline{l}j}^{\overline{r}}=0,} \\ 
^{(M)}\widehat{{\Gamma }}{_{l\overline{j}}^{\overline{r}}=}%
\tsum\limits_{\lambda =1}^{p}{\Gamma _{ls}^{v_{\lambda }}\delta
_{r_{1}}^{j_{1}}...\delta _{r_{q}}^{j_{q}}\delta _{i_{1}}^{l_{1}}...\delta
_{i_{\lambda }}^{s}...\delta _{i_{p}}^{l_{p}}-}\tsum\limits_{\mu =1}^{q}{%
\Gamma _{lr_{\mu }}^{s}\delta _{r_{1}}^{j_{1}}...\delta _{s}^{j_{\mu
}}...\delta _{r_{q}}^{j_{q}}\delta _{i_{1}}^{l_{1}}...\delta
_{i_{p}}^{l_{p}},} \\ 
^{(M)}{}\widehat{{\Gamma }}{_{lj}^{r}=\Gamma _{lj}^{r}+\frac{1}{2f}{}%
^{f}A_{lj}^{r},}%
\end{array}%
\right.
\end{equation*}%
with respect to the adapted frame.
\end{proposition}

By using the components $^{(M)}{}\widehat{{\Gamma }}_{\gamma \beta }^{\alpha
}$ of the metric connection $^{(M)}\widehat{\nabla }$ in (\ref{B6.1}), we get%
\begin{equation}
\left\{ 
\begin{array}{c}
\frac{d^{2}x^{r}}{dt^{2}}+({\Gamma _{lj}^{r}+\frac{1}{2f}{}^{f}A_{lj}^{r})}%
\frac{dx^{l}}{dt}\frac{dx^{j}}{dt}=0, \\ 
\frac{d}{dt}(\frac{\delta t_{r_{1}...r_{q}}^{v_{1}...v_{p}}}{dt}%
)+(\sum\limits_{\lambda =1}^{p}\Gamma _{ls}^{v_{\lambda }}\delta
_{r_{1}}^{j_{1}}...\delta _{r_{q}}^{j_{q}}\delta _{i_{1}}^{l_{1}}...\delta
_{i_{\lambda }}^{s}...\delta _{i_{p}}^{l_{p}}- \\ 
-\sum\limits_{\mu =1}^{q}\Gamma _{lr_{\mu }}^{s}\delta
_{r_{1}}^{j_{1}}...\delta _{s}^{j_{\mu }}...\delta _{r_{q}}^{j_{q}}\delta
_{i_{1}}^{l_{1}}...\delta _{i_{p}}^{l_{p}})\frac{dx^{l}}{dt}\frac{\delta
t_{j_{1}...j_{q}}^{i_{1}...i_{p}}}{dt}=0,%
\end{array}%
\right.  \label{B6.7}
\end{equation}%
where the second equation of (\ref{B6.7}) reduces to%
\begin{equation*}
\frac{\delta ^{2}t_{r_{1}...r_{q}}^{v_{1}...v_{p}}}{dt^{2}}=0.
\end{equation*}%
Thus we have the last result.

\begin{theorem}
Let $\widetilde{\gamma }$ be a geodesic on $T_{q}^{p}(M)$ with respect to
the metric connection $^{(M)}\widehat{\nabla }$ of $^{S}g_{f}$. Then the
projection $\gamma $ of $\widetilde{\gamma }$ is a geodesic with respect to
the Levi-Civita connection $\nabla $ on $M$ and the tensor field $%
t_{r_{1}...r_{q}}^{v_{1}...v_{p}}(t)$ defined along $\gamma $ has vanishing
second covariant derivative if and only if $f=C(const.).$\bigskip
\end{theorem}

\end{document}